\documentclass [12pt]{amsart}
\usepackage{amssymb,latexsym,amsfonts, amsthm, amsmath, amsrefs, mathtools,color}
\usepackage{array}
\usepackage{enumerate}
\usepackage{booktabs}
\usepackage{fancyhdr}
\usepackage{verbatim}
\usepackage[margin=1in,twoside=false]{geometry}
\usepackage{multicol}
\usepackage{longtable}
\usepackage{hyperref}
\usepackage{mathrsfs}
\usepackage{tikz}
\usetikzlibrary{matrix,arrows}
\usepackage[capitalize]{cleveref}
\usepackage{tikzcd}
\crefname{figure}{Figure}{Figures}
\crefname{prop}{Proposition}{Propositions}

\theoremstyle{definition}

\newtheorem{thm}{Theorem}[section]   
\newtheorem{cor}[thm]{Corollary}     
\newtheorem{lem}[thm]{Lemma}         
\newtheorem{prop}[thm]{Proposition}  
\newtheorem{obs}[thm]{Observation}  

\theoremstyle{definition} 
\newtheorem{defn}[thm]{Definition}   
\newtheorem*{conj*}{Conjecture}        
\newtheorem{ex}[thm]{Example}        

\newtheorem{rmk}[thm]{Remark}
\newtheorem{question}[thm]{Question}

\DeclareMathOperator{\code}{code}

\DeclareMathOperator{\link}{link}
\DeclareMathOperator{\id}{id}

\DeclareMathOperator{\tk}{Tk}

\DeclareMathOperator{\conv}{conv}
\DeclareMathOperator{\aff}{aff}

\DeclareMathOperator{\Code}{\mathbf{Code}}

\DeclareMathOperator{\OM}{\mathbf{OM}}
\DeclareMathOperator{\OMatroid}{\mathbf{OM}}

\DeclareMathOperator{\NRing}{\mathbf{NRing}}
\DeclareMathOperator{\Ring}{\mathbf{Ring}}
\DeclareMathOperator{\OMRing}{\mathbf{OMRing}}

\newcommand{\lop}{\circ}

\def\D{\mathbb D}

\def\F{\mathbb F}

\def\R{\mathbb R}

\def\cC{\mathcal C}

\def\cF{\mathcal F}

\def\cH{\mathcal H}

\def\cL{\mathcal L}
\def\cM{\mathcal M}

\def\cP{\mathcal P}
\def\cQ{\mathcal Q}

\def\cT{\mathcal T}
\def\cU{\mathcal U}
\def\cV{\mathcal V}
\def\cW{\mathcal W}

\def\scC{\mathscr C}
\def\scD{\mathscr D}
\def\scE{\mathscr E}

\def\scH{\mathscr H}

\def\scM{\mathscr M}

\def\sfD{\mathsf{D}}

\def\sfL{\mathsf{L}}

\def\sfR{\mathsf{R}}
\def\sfS{\mathsf{S}}

\def\sfW{\mathsf{W}}

\def\a{\alpha}
\def\b{\beta}

\def\D{\Delta}

\def\s{\sigma}
\def\t{\tau}

\def\inv{^{-1}}

\def\sm{\setminus}

\newcommand{\ang}[1]{\left\langle #1\right\rangle}	

\renewcommand{\emptyset}{\varnothing}

\def\lra{\leftrightarrow}

\def\1{\overline{1}}
\def\2{\overline{2}}
\def\3{\overline{3}}
\def\4{\overline{4}}
\def\5{\overline{5}}

\def\bx{\mathbf{x}}
\def\by{\mathbf{y}}

\def\id{\operatorname{id}}

\def\uf{\underline{f}}

\graphicspath{{figures/}}

\DeclareMathOperator{\sep}{sep}

\newcommand{\pcode}{\mathbf P_{\mathbf {Code}}}

\newtheorem{introtheorem}{Theorem}

\title{Oriented Matroids and Combinatorial Neural Codes}
\date{\today}
\author{Alexander Kunin$^*$} \thanks{$^*$Department of Mathematics, Creighton University, alexkunin@creighton.edu}
\author{Caitlin Lienkaemper$^\dagger$}\thanks{$^\dagger$Department of Mathematics and Statistics, Boston University, clienk@bu.edu} 
\author{Zvi Rosen$^\ddagger$}\thanks{$^\ddagger$Department of Mathematics, Florida Atlantic University, rosenz@fau.edu} 
\begin{document}
\maketitle

\begin{abstract}
A combinatorial neural code $\scC\subseteq 2^{[n]}$ is called convex if it arises as the intersection pattern of convex open subsets of $\R^d$.  
We relate the emerging theory of convex neural codes to
the established theory of oriented matroids, both with respect to geometry and computational complexity and categorically. 
For  geometry and computational complexity, we show that a code has a realization with convex polytopes if and only if it lies below the code of a representable oriented matroid in the partial order of codes introduced by Jeffs.
We show that previously published examples of non-convex codes do not lie below any oriented matroids, and we construct examples of non-convex codes lying below
non-representable oriented matroids. 
By way of this construction, we can apply Mn\"{e}v-Sturmfels universality to show that deciding
whether a combinatorial code is convex is NP-hard.

On the categorical side, we show that the map taking an acyclic oriented matroid to the
code of positive parts of its topes is a faithful functor.
We adapt the oriented matroid ideal introduced by Novik, Postnikov, and Sturmfels into
a functor from the category of oriented matroids to the category of rings;
then, we show that the resulting ring
maps naturally to the neural ring of the matroid's neural code.
\end{abstract}

\setcounter{tocdepth}{1}
\tableofcontents


\section{Introduction}

A combinatorial neural code is a collection $\scC$ of subsets of $[n] := \{1,\dots,n\}$.
Such codes model neural activity, with each codeword $\s\subseteq[n]$ in $\scC$ representing a set of neurons which are simultaneously active.
Our motivating example is the activity of hippocampal place cells, neurons in the brain which encode an animal's physical location \cite{okeefe1978}.
Each neuron $i$ is active when the animal is in a corresponding subset $U_i$ of the animal's environment $X \subseteq \R^d$, called the $i$-th \emph{place field}.
If neural activity is viewed as a function $X \to \{0,1\}^n$, then the set $U_i$ is the support of the $i$-th component of this function, the ``preferred location'' of that neuron.
Neurons acting as indicators for a preferred set of stimuli appear across many sensory modalities, so we will use the slightly broader term \emph{receptive field} to de-emphasize the spatial navigation aspect.

In this simplified model, neurons fire together if and only if their receptive fields overlap,
and thus the code represents the intersection pattern of the receptive fields. 
This information can reveal significant topological and geometric information in experimental data, such as the topology of an animal's environment \cite{curto2008cellgroups} or the intrinsic geometry of more abstract stimulus spaces \cite{giusti2015clique,zhou2018hyperbolic}.
Receptive fields are often observed to be convex, and therefore
we are interested in characterizing \emph{convex} neural codes: codes that arise as the intersection patterns of convex open subsets of some Euclidean space.
For example, \cref{fig:code_ex} illustrates three convex receptive fields and the associated convex code.

\begin{figure}[h!]
	\includegraphics[width = 3 in]{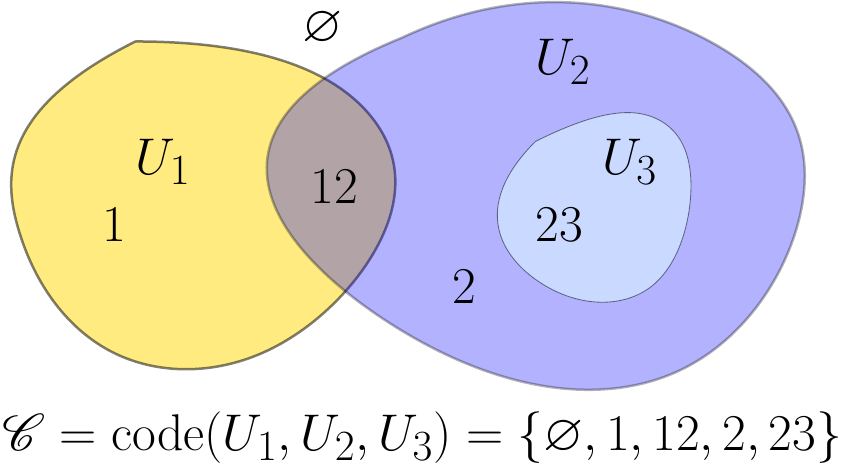}
	\caption{The code of $U_1, U_2, U_3$ is $\scC = \{\emptyset,1,12,2, 23\}$.
	\label{fig:code_ex}} 
\end{figure}

Beyond experimental motivation, requiring receptive fields to be convex yields rich theoretical results.
In particular, the nerve lemma can be used to deduce topological properties of simplicial complexes associated to convex codes \cite{curto2017makes, chen2019neural}. 
Another useful tool developed to study neural codes is the {\em neural ring} \cite{curto2013neural}, the coordinate ring of the code as an algebraic variety in $\F_2^n$.
This was used to detect obstructions to convexity in \cite{curto2019algebraic}.
However, there are many examples of non-convex codes which cannot be captured by these obstructions \cite{lienkaemper2017obstructions, jeffs2019sunflowers}.
While other classes of neural codes have been completely characterized (e.g. codes described by connected receptive fields \cite{mulas2020minimal}, or convex codes on five or fewer neurons \cite{goldrup2019classification}), convex codes have evaded full description.

As the literature on combinatorial neural codes proliferated, we observed various similarities with the well-studied realm of oriented matroid theory.
For instance, the class of stable hyperplane codes introduced in \cite{itskov2018hyperplane} are defined by a collection of half-spaces intersecting a convex set, which are precisely the sets of topes of a realizable COM (conditional oriented matroid) as studied in \cite{bandelt2018coms}.  
The neural ideal, defined in \cite{curto2013neural} and further developed in \cite{curto2019algebraic, gunturkun2017polarization, de2019neural}, seems to align with the oriented matroid ideal defined in \cite{novik2002syzygies}, particularly after the neural ideal is polarized \cite{gunturkun2017polarization}.
Finally, morphisms of codes, as defined in \cite{jeffs2019morphisms}, seem analogous to strong maps of oriented matroids, as formulated in \cite{hochstattler1999linear}. 
In this paper, we draw parallels between the notions of convexity for neural codes and representability for oriented matroids, and formalize these connections on a functorial level.

\begin{figure}[h!]
  \includegraphics[width = 5 in]{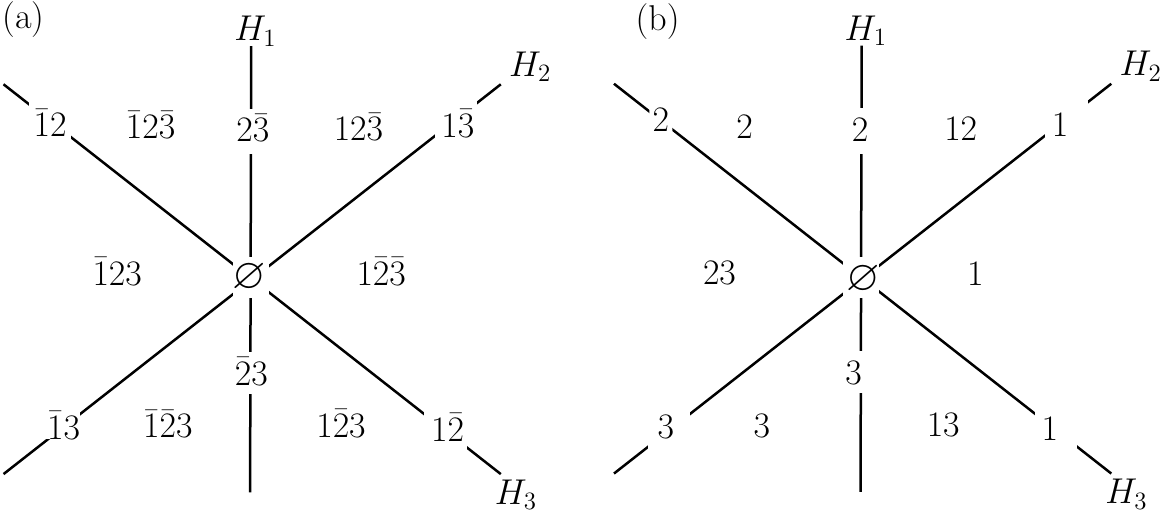}
  \caption{(a)~The covectors of an oriented matroid arising from a central hyperplane arrangement. (b)~The combinatorial code of the cover given by the positive open half-spaces.}
  \label{F:OMandCodeExample}
\end{figure}

First, we establish strong connections between representable oriented matroids and convex neural codes 
by considering the map $\sfL^+: \OM \to \Code$ which takes an oriented matroid to the positive parts of its covectors. Representable oriented matroids are precisely those which can be obtained from real hyperplane arrangements, as in \cref{F:OMandCodeExample}(a).
Isomorphism classes of neural codes form a partially ordered set denoted $\pcode$, introduced in \cite{jeffs2019morphisms}.
Roughly, $\scC\leq \scD$ if there is a way to construct a realization for $\scC$ using a realization of $\scD$, in which case we say $\scC$ is a \emph{minor} of $\scD$. 
%
We extend a result of \cite{jeffs2019morphisms} and use it to prove that
all codes lying below codes of representable oriented matroids
have realizations with interiors of convex polytopes, which we call \emph{open polytope convex}. Further, the converse also holds: {
	\begin{introtheorem}\label{thm:polytope_matroid}
		A code is open polytope convex if and only if it is a minor of $\sfL^+\cM$ for some representable oriented matroid $\cM$.
	\end{introtheorem}
	Since every code which has a convex realization in the plane has a realization with convex polygons \cite{bukh2022planar}, this result gives a full characterization of planar convex codes in terms of representable oriented matroids. In higher dimensions, it is not yet known whether every convex code has a realization with convex polytopes.
	 If this does hold, then \cref{thm:polytope_matroid} would give a full characterization of convex codes in terms of representable oriented matroids. 
Without this result, we can still categorize non-convex codes: if a code is not convex, then either it does not lie below any oriented matroid in $\pcode$, or it lies below non-representable matroids only. 
There are many known examples of non-convex codes \cite{curto2017makes, lienkaemper2017obstructions,jeffs2019morphisms, jeffs2019sunflowers, chen2019neural},
and we show that many of these fall into the first category: they are non-convex because they are not below any oriented matroids in $\pcode$.
For instance, codes with topological local obstructions do not lie below oriented matroids. 
Furthermore, sunflower codes \cite{jeffs2019sunflowers}, a well known family of non-convex codes with no local obstructions also do not lie below oriented matroids. 
	\begin{introtheorem}\label{thm:bad_sunflower}
		The non-convex codes with no local obstructions constructed via sunflowers do not lie below codes of oriented matroids in $\pcode$.
	\end{introtheorem}

We are also able to generate an infinite family of non-convex codes of the second kind, those which lie below non-representable matroids only.
In order to obtain this family, we establish a relationship between representability and convexity. We do this for the special case of uniform oriented matroids of rank 3, which correspond to non-degenerate pseudoline arrangements in the plane. 
\begin{introtheorem}\label{thm:witness}
\
	 Let $\cM$ be a uniform, rank 3 oriented matroid. Then we can construct a code which is convex if and only if $\cM$ is representable. 
\end{introtheorem}
Using this result, we are able to compare two fundamental decision problems: (1) is a given oriented matroid representable, and (2) is a given neural code realizable by convex sets. 
We demonstrate that deciding convexity for arbitrary neural codes is {\em at least} as hard as deciding representability of an oriented matroid. 
The latter problem is known to be NP-hard and $\exists\R$-hard, leading to the following theorem:

\begin{introtheorem}\label{thm:hard}
	The convex code decision problem is NP-hard and $\exists\R$-hard.
\end{introtheorem}

Finally, we relate algebraic and categorical structures for matroids and codes. Acyclic oriented matroids form a category  $\OMatroid$ whose morphisms are given by strong maps, as defined in \cite{hochstattler1999linear}. Neural codes form a category  $\Code$ introduced in \cite{jeffs2019morphisms}, with morphisms defined using trunks of codes. 
We show that the map $\sfW^+: \OMatroid \to \Code$ which takes an acyclic oriented matroid to the positive parts of its topes is a faithful functor.
Furthermore, we adapt the oriented matroid ideal introduced in \cite{novik2002syzygies} to non-affine oriented matroids, producing the oriented matroid dual ideal $O(\cM)^\star$ and the oriented matroid ring $k[x_1,\dots,x_n,y_1,\dots,y_n] / O(\cM)^\star$.
We show that the map $\sfS$ taking an oriented matroid to its oriented matroid ring is a functor, and use this to define the category $\OMRing$.
Using results from \cite{gunturkun2017polarization}, we define the depolarization map $\sfD: \OMRing\to \NRing$, and show that this map is functorial. Finally, we show that these maps play nicely with the functor $\sfR: \Code \to \NRing$ from \cite{jeffs2019morphisms}. 

 \begin{introtheorem}\label{thm:foc}
The maps $\sfS$, $\sfD$, and $\sfW^+$ are functorial. In particular, the map $\sfW^+$ is a faithful, but not full functor from $\OMatroid \to \Code$. 
 Moreover, the square below commutes, that is, $\sfR \circ\sfW^+ = \sfD\circ \sfS$. \begin{center}
\begin{tikzcd}
\OMatroid  \arrow[r, "\sfS"] \arrow[d, "\sfW^+"]
& \OMRing \arrow[d, "\sfD"] \\
\Code \arrow[r,  "\sfR"]
& \NRing
\end{tikzcd}
\end{center}
\end{introtheorem}

The paper is organized as follows: In \cref{S:background}, we establish notation and background material that will be necessary for later sections.
In \cref{sec:intersection}, we relate minors of codes to realizability by intersection-closed families of sets, and use this to establish Theorem 2.
In \cref{sec:non_convex}, we discuss classes of non-convex codes and their relationships to oriented matroids. 
In \cref{sec:algebra}, we detail the functors among the categories of acyclic oriented matroids, combinatorial neural codes, and rings. 
Finally, in \cref{S:questions}, we present open questions related to each area discussed in the paper.



\section{Background}\label{S:background}

We provide the essential background information on oriented matroids in \cref{sec:OMbackground} and combinatorial codes in \cref{sec:CCbackground}.
In \cref{sec:codesfromOMs}, we define the maps $\sfW^+$ and $\sfL^+$ which take oriented matroids to combinatorial codes.
This section is by no means comprehensive, and we will occasionally make reference to theorems not fully stated in the text. Our primary reference for oriented matroids is \cite{bjorner1999oriented}, which collects and synthesizes decades of research on the subject. For combinatorial codes, there is no single comprehensive reference; for those unfamiliar with the subject, \cite{curto2013neural} and \cite{jeffs2019morphisms} are good introductions to the neural ring and morphisms of codes, respectively.

\subsection{Oriented matroids}\label{sec:OMbackground}
An oriented matroid $\cM = (E,\cL)$ consists of a finite ground set $E$
and a collection $\cL \subseteq 2^{\pm E}$ of \emph{signed subsets} of $\pm E$ satisfying certain axioms.
Typically, we will take $E = [n] := \{1, \ldots, n\}$, $\bar E = [\bar n] := \{\bar 1, \ldots, \bar n\}$, and $\pm E := E \cup \bar{E}$.
The set $\pm E$ is endowed with the involution $-: \pm E \to \pm E$,
exchanging $e \in E$ with $\bar e \in \bar E$.
The negative of a subset $X \subseteq \pm E$ is $-X := \{-x \mid x\in X\}$.
The \emph{support} of a set $X\subseteq \pm E$ is the set
$\underline X := \{ e\in E \mid e\in X \mbox{ or } -e \in X \} \subseteq E$. 
The \emph{positive part} of $X$ is $X^+ := X \cap E$ and the
\emph{negative part} is $X^- := (-X) \cap E$. 

A set $X\subseteq \pm E$ is a \emph{signed set} if its positive and negative parts are disjoint.
If $e\in E$ and $X$ is a signed subset of $\pm E$, define $X_e$ by $X_e = +$ if $e\in X$, $X_e = -$ if $-e\in X$, and $X_e = 0$ otherwise; in this way, we can consider signed subsets equivalently as subsets of $\pm E$ or as vectors in $\{+,0,-\}^E$.
The \emph{composition} of sign vectors $X$ and $Y$ is defined component-wise by
\begin{align*}
  (X\circ Y)_e := 
  \begin{cases}
    X_e \mbox{ if } X_e\neq 0\\
    Y_e \mbox{ otherwise}.
  \end{cases}
\end{align*}
The \emph{separator} of $X$ and $Y$ is the unsigned set $\sep(X, Y) := \{ e\mid X_e = -Y_e\neq 0\}$. 

Now, we are ready to 	define oriented matroids, which we do via the covector axioms. 	
\begin{defn}
  Let $E$ be a finite set, and $\cL \subseteq 2^{\pm E}$ a collection of signed subsets satisfying the following
  \emph{covector axioms}:
\begin{enumerate}[(V1)]
\item\label{axiom:emptysetV} $\emptyset \in \mathcal L$
\item\label{axiom:symmetryV} $X\in \mathcal L $ implies $-X\in \mathcal L$. 
\item\label{axiom:compositionV} $X, Y\in \mathcal L$ implies $X\circ Y\in \mathcal L$. 
\item\label{axiom:crossingV} If $X, Y\in \mathcal L$ and $e \in \sep(X,Y)$, then there exists $Z\in \mathcal L$ such that $Z_e = 0$ and $Z_f = (X\circ Y)_f = (Y\circ X)_f$ for all $f\notin \sep(X, Y)$. 
\end{enumerate}
Then, the pair $\cM = (E, \cL)$ is called an \emph{oriented matroid}, and $\cL$ its set of covectors.
\end{defn}

Maximal covectors (with respect to inclusion) are called \emph{topes}.
An oriented matroid is \emph{acyclic} if it has a positive tope, i.e.\ a tope with empty negative part.

\begin{ex}
  A central hyperplane arrangement $\cH$ in $\R^d$ produces an oriented matroid.
  Let $\ell_1,\dots,\ell_n$ be linear forms on $\R^d$, and $H_1,\dots,H_n$ their zero sets (i.e.\ hyperplanes).
  We can assign each point $x\in \R^d$ to a signed set $X\subseteq \pm [n]$ by 
  \begin{align*}
    X_i = \begin{cases}
      + \mbox{ if }  \ell_i(x)  > 0\\
      - \mbox{ if }  \ell_i(x)  < 0\\
      0 \mbox{ if }  \ell_i(x)  = 0.
    \end{cases}
  \end{align*}
  The family of signed sets which arise in this way satisfies the covector axioms,
  and therefore defines an oriented matroid.
  Notice that each covector corresponds to a cell of the hyperplane arrangement, and that topes correspond to top-dimensional cells. 
  We will refer to this oriented matroid $\cM(\cH) = ([n], \cL(\cH))$ as the oriented matroid of $\cH$.

  An oriented matroid $\cM$ is \emph{representable} if $\cM = \cM(\cH)$ for some hyperplane arrangement $\cH$. 
		\cref{F:OMandCodeExample}(a) illustrates an example in $\R^2$.
		 The \emph{rank} of a representable oriented matroid is the minimum possible dimension of the hyperplane arrangement $\cH$ with $\cM = \cM(\cH)$.
		Not every oriented matroid is representable. However, we are able to take this hyperplane picture as paradigmatic. The topological representation theorem guarantees that every oriented matroid has a representation by a pseudosphere arrangement \cite{folkman1978oriented}. For details, see \cite[Chapter 5]{bjorner1999oriented}. 
	\end{ex}
	
We can also use oriented matroid theory to describe affine hyperplane arrangements. 
An \emph{affine} oriented matroid consists of a pair $(\cM, g)$ where $\cM$ is an oriented matroid and $g$ is a distinguished element of its ground set. 
The \emph{affine space} of $(\cM, g)$ is the set $\{X \in \cL(\cM) \mid X_g = +\}$. 
We can embed an affine hyperplane arrangement in $\R^d$ as a central hyperplane arrangement in $\R^{d+1}$. 
We replace each affine hyperplane $w_1 x_1 + \cdots + w_dx_d +  b =0$ with a central hyperplane $w_1 x_1 + \cdots + w_dx_d +  bx_{d+1}=0$, and add an additional hyperplane $x_{d+1} = 0$. 
Restricting to the $x_{d+1}=1$ plane recovers the original hyperplane arrangement. 
Each chamber of the affine arrangement corresponds to an element of the affine space of $(\cM(\cH), g)$, where $g$ corresponds to the hyperplane $x_{d+1} = 0$. Because we can translate between affine and central realizations, the matroid $\cM$ is representable with a central hyperplane arrangement if and only if $(\cM, g)$ is representable by an affine hyperplane arrangement for any ground set element $g$. 
	

There are many equivalent axiomatizations of oriented matroids.
The two formulations we use most often throughout this work are the covector axioms (V\ref{axiom:emptysetV})-(V\ref{axiom:crossingV}), stated above, and the circuit axioms (C\ref{axiom:emptysetC})-(C\ref{axiom:weakelimC}), which we state here.

\begin{defn}\label{D:circuitaxioms}
  Let $E$ be a finite set, and $\cC\subseteq 2^{\pm E}$ a collection of signed subsets satisfying the following \emph{circuit axioms}: 
\begin{enumerate}[(C1)]
	\item\label{axiom:emptysetC} $\varnothing \notin \cC$.
	\item\label{axiom:symmetryC} $X \in \cC$ implies $-X \in \cC$.
	\item\label{axiom:incomparableC} $X,Y \in \cC$ and $\underline X \subseteq \underline Y$ implies $X = Y$ or $X = -Y$.
	\item\label{axiom:weakelimC} For all $X,Y \in \cC$ with $X \neq -Y$ and an element $e \in X^+ \cap Y^-$, there is a $Z \in \cC$ such that $Z^+ \subseteq (X^+ \cup Y^+) \setminus e$ and $Z^- \subseteq (X^- \cup Y^-) \setminus e$.
\end{enumerate}
Then the pair $\cM = (E, \cC)$ is an oriented matroid, and $\cC$ is its set of circuits. 

 In some contexts, we admit the sets $\{e, \bar{e}\}$ as \emph{improper circuits}. We will call a circuit a \emph{proper circuit} when we wish to emphasize that it is a signed set, i.e.\ its positive and negative parts are disjoint.
\end{defn}

For an oriented matroid arising from a hyperplane arrangement, circuits are the signs of the coefficients in the minimal linear dependencies among the normal vectors to the hyperplanes.
As a result, the rank of an oriented matroid can be defined via its circuits -- the rank of an oriented matroid is the one less than the cardinality of its largest circuit. 
An oriented matroid is \emph{uniform} of rank $r$ if all of its circuits have the same cardinality of $r+1$. Uniform oriented matroids correspond to generic hyperplane arrangements. 

An element $e \in E$ is a \emph{loop} of $\cM$ if $\{e\} \in \cC(\cM)$. An oriented matroid is \emph{loopless} if no element is a loop.

Proper circuits are related to covectors as follows: 
Two signed sets $X$ and $Y$ are called \emph{orthogonal} if either $\underline X \cap \underline Y = \emptyset $ or if there exist $e, f\in \underline X \cap \underline Y $ such that $X_eX_f  = - Y_e Y_f$. 
A signed set is called a \emph{vector} of $\cM$ if and only if it is orthogonal to every covector.
Equivalently, a signed set is a vector of $\cM$ if and only if it is orthogonal to every tope.
The circuits are the minimal vectors of $\cM$.
The vectors of an oriented matroid define a dual oriented matroid, hence the vector and covector axioms are identical.
Minimal covectors are know as \emph{cocircuits}. 
They also satisfy the circuit axioms. 
		
For a given oriented matroid $\cM$, each one of the set of covectors $\cL(\cM)$, the set of topes $\cW(\cM)$, the set of vectors $\cV(\cM)$, and the set of circuits $\cC(\cM)$ is sufficient to recover all of the others.

\subsection{Combinatorial codes}\label{sec:CCbackground}
A combinatorial code $\scC$ is a collection of subsets of a finite set $V$, i.e.\ $\scC \subseteq 2^{V}$.
Typically, we take $V = [n]$.
	
Given an arbitrary set $X$ and collection $\cU = \{U_1,\dots,U_n\}$ with each $U_i \subseteq X$,
the \emph{code of the cover $\cU$ (relative to $X$)} is
	\[ \code(\cU,X) := \Biggl\{ \sigma \subseteq [n] \Biggm\vert \bigcap_{i\in\sigma} U_i \setminus \bigcup_{j\notin\sigma} U_j \neq \varnothing \Biggr\}. \] 
Note we do not require $X = \bigcup_{i\in[n]} U_i$; indeed, $\varnothing \in \code(\cU,X)$ if and only if $\bigcup_{i\in[n]} U_i \subsetneq X$.
A code $\scC$ is called \emph{open convex} if there exists a collection $\cU = \{U_1,\dots,U_n\}$ of open convex sets and an open convex set $X \subseteq \R^d$, such that $\scC = \code(\cU,X)$, for some $d$. We will refer to open convex codes simply as convex codes. 
	
\begin{ex}
  Let $U_i$ denote the open half-space on the positive side of hyperplane $H_i$ in \cref{F:OMandCodeExample}(a), i.e.\ $U_i = \{x \in \R^2 \mid \ell_i(x) > 0\}$. Then, $\code(\cU,\R^n)$ is the combinatorial code with codewords as labeled in \cref{F:OMandCodeExample}(b).
	\end{ex}
	
		
		
Morphisms of combinatorial codes were defined in \cite{jeffs2019morphisms} in terms of trunks.
For $\s\subseteq[n]$, the \emph{trunk of $\s$ in $\scC$} is the set of codewords which contain $\s$,
\[ \tk_\scC(\sigma) := \{\tau\in \scC \mid \sigma \subseteq \tau\}. \]
A subset of $\scC$ is a trunk if it is equal to $\tk_\scC(\sigma)$
for some $\sigma \subseteq [n]$ or if it is empty.
A \emph{simple trunk} is the trunk of a singleton set.
A map $f: \scC \to \scD$ is a \emph{morphism of codes} if the preimage of each
trunk of $\scD$ is a trunk of $\scC$. Any set of trunks $T_1, \ldots, T_m\subseteq\scC$ defines a morphism $f:\scC \to 2^{[m]}$ by $f(\sigma) := \{i \mid \sigma\in T_i\}$, and any code morphism $f: \scC\to \scD$ can be obtained in this way \cite[Proposition 2.12]{jeffs2019morphisms}.
The class of codes, together with these morphisms, forms the category $\Code$. 
		
	
A subset $\sigma \subseteq [n]$ can be encoded as a point $c \in \F_2^n$ by setting
$c_i = 1$ for $i \in \sigma$ and $c_i = 0$ for $i \notin \sigma$.
Hence, a code $\scC \subseteq 2^{[n]}$ can equivalently be thought of as a variety
$\scC \subseteq \F_2^n$.
The \emph{vanishing ideal} of a code $\scC$ is the ideal
	\[ I_\scC := \{ f(\mathbf{x}) \in \F_2[x_1,\dots,x_n] \mid f(c) = 0 \mbox{ for all } c \in \scC\}, \]
and the \emph{neural ring} of $\scC$ is the quotient ring $R_\scC = \F_2[x_1,\dots,x_n] / I_\scC$.
The vanishing ideal $I_\scC$ is a \emph{pseudo-monomial ideal}, meaning it is generated by products of the form $\prod_{i\in\s}x_i \prod_{j\in\t}(1-x_j)$, called pseudo-monomials.
As with circuits, we distinguish between \emph{proper} pseudo-monomials, with $\s$ and $\t$ disjoint, and \emph{improper} pseudo-monomials, which are divisible by some $x_i(1-x_i)$.
We will briefly discuss the vanishing ideal and neural ring in \cref{S:neuralring}, but
many more details can be found in \cite{curto2013neural}.
For clarity, we will denote pseudo-monomials and monomials with superscripts, i.e.\
	\[ x^\s (1-x)^\t := \prod_{i\in\s} x_i \prod_{j\in\t}(1-x_j) \qquad \mbox{and} \qquad x^\s y^\t := \prod_{i\in\s} x_i \prod_{j\in\t}y_j. \]

\subsection{Codes from oriented matroids}\label{sec:codesfromOMs}

Consider an oriented matroid $\cM$ on ground set $E$.
We can consider the positive parts of covectors (respectively, topes) as a code on $E$, which we denote $\sfL^+\cM$ (respectively $\sfW^+\cM$) to emphasize the change in categories:
	\[
		\sfL^+\cM := \{ X^+ \mid X \in \cL(\cM)\} \qquad \text{and} \qquad
		\sfW^+\cM := \{ W^+ \mid W \in \cW(\cM)\}. \]
In \cref{sec:intersection,sec:non_convex} we will show how $\sfL^+$ relates representability of oriented matroids with convexity of codes.
In \cref{sec:algebra}, we will examine the functorial properties of $\sfW^+$; culminating in a proof of \cref{thm:foc}.



If $\cM$ is the matroid of a hyperplane arrangement the code $\sfL^+\cM$ matches the code of the cover given by positive sides of the hyperplanes (as in \cref{F:OMandCodeExample}). 
This extends to any topological representation of $\cM$ by a pseudosphere arrangement (as introduced in \cite{folkman1978oriented}).
\begin{obs}\label{prop:covectors_match_cover}
	If  $\cM$ is any oriented matroid and $\{S_e\}_{e\in E}$ is an oriented pseudosphere arrangement topologically realizing $\cM$, then 
		\[\sfL^+\cM = \code(\{S_e^{+}\}_{e\in E}, \mathbb S^d). \]
		
	In particular, if $\cM$ is a representable oriented matroid and $\{H_e\}_{e\in E}$ is an oriented hyperplane arrangement realizing $\cM$, then 
		\[ \sfL^+\cM = \code(\{H_e^{+}\}_{e\in E},  \mathbb R^{d}). \]
\end{obs}

The map $\sfL^+$ is better behaved geometrically than $\sfW^+$. In particular, \cref{prop:covectors_match_cover} fails for $\sfW^+$. For instance, in the hyperplane arrangement pictured in \cref{F:OMandCodeExample}, $\emptyset$ is a codeword in the code of the cover given by the positive open half-spaces, but is not the positive part of any tope. 

\begin{rmk}
	If $\cM$ is an acyclic oriented matroid, then the the signed set $-E$ is a tope.
	Thus, if $S \subseteq E$ is the positive part of some covector $X\in \cL(\cM)$, then $S$ is also the positive part of the tope $X\circ -E\in \cW(\cM).$
	Thus, on acyclic oriented matroids, $\sfW^+$ and $\sfL^+$ coincide.
\end{rmk}


 \section{Intersection-closed families and morphisms}
\label{sec:intersection}
In \cite{jeffs2019morphisms}, Jeffs shows that the image of a convex code under a code morphism, as well as any trunk of a convex code, is itself a convex code. From this observation, he defines the poset of isomorphism classes of codes $\pcode$ in which convex codes form a down-set:
if $\scD \leq \scC$ in $\pcode$ and $\scC$ is convex, then so is $\scD$.
In this section, we generalize this statement to intersection-closed families,
of which open convex subsets of $\R^d$ is one example.
A family $\cF$ of subsets of a topological space is called {\em intersection-closed} if it is closed under finite intersections and contains the empty set. 
We say that a neural code $\scC$ is {\em $\cF$-realizable} if $\scC= \code(\cU, X)$ for some $\cU \subseteq \cF$ and $X\in \cF$.
For instance, a neural code is convex if and only if it is $\cF$-realizable for the set $\cF$ of convex open subsets of some $\R^d$. 
Then, using this result, we prove \cref{thm:polytope_matroid}.

We recall the relevant definitions.
Two codes $\scC$ and $\scD$ are isomorphic if there is a bijective code morphism $f:\scC \to \scD$
whose inverse is also a code morphism.
If there is a sequence of codes $\scC = \scC_0,\scC_1,\dots,\scC_k=\scD$ such that each successive code is either the image of a morphism from or a trunk of the preceding code, we say $\scD$ is a \emph{minor} of $\scC$ \cite{jeffs2021embedding}.
Codes are then quasi-ordered by setting $\scD \leq \scC$ if $\scD$ is a minor of $\scC$.
The poset of isomorphism classes of codes induced by this order is denoted $\pcode$.

\begin{prop} \label{lem:int_closed}
	For any intersection closed family $\cF$,
	if $\scC$ is $\cF$-realizable and  $\scD \leq \scC$, then $\scD$ is $\cF$-realizable. 
\end{prop}

\begin{proof}
We first check the case $\scD = f(\scC)$.  This closely follows the proof of Theorem 1.4 in \cite{jeffs2019morphisms},  since the only property of convex sets this proof uses is that the family of open convex subsets of $\R^d$ is closed under finite intersection.  We repeat the details here. Let $\scC\subseteq 2^{[n]}$, $\scD\subseteq 2^{[m]}$,  and $T_1, \ldots, T_m$ be the trunks in $\scC$ that define the morphism $f: \scC \to \scD$. Let $\{U_1, \ldots, U_n\} \subseteq \cF$ be an $\cF$-realization of $\scC$. 

If $T_j$ is nonempty, let $\sigma_j$ be the intersection of all elements of $T_j$, which is the unique largest subset of $[n]$ such that $T_j = \tk_\scC(\sigma_j)$.
Then, for $j\in [m]$, define 
\begin{align*}
V_j = \begin{cases}
\emptyset & T_j = \emptyset\\
\bigcap_{i\in \sigma_j} U_i & T_j \neq \emptyset
\end{cases}
\end{align*}
Since $\cF$ is closed under finite intersection and contains the empty set, $V_j\in \cF$ for all $j\in [m]$. Thus, it suffices to show that the code $\scE$ that they realize is $\scD$. To see this, note that we can associate each point $p\in X$ to a codeword in $\scC$ or $\scE$ by $p\mapsto \{i\in [n] \mid p\in U_i\}$ and $p\mapsto\{j\in [m] \mid p\in V_j\}$. Then let $p\in X$ be arbitrary, and let $c$ and $e$ be the associated codewords in $\scC$ and $\scE$ respectively. By the definition of the $V_j$, we have that $c\in T_j$ if and only if $j\in e$, which is equivalent to $e = f(c)$. Since $p$ was arbitrary and every codeword arises at some point, we conclude that $\scE = f(\scC) = \scD$, as desired. 

Next, we check the case $\scD = \tk_{\scC} (\sigma)$.  In this case, let $\scC \subseteq 2^{[n]}$, $\sigma \subseteq [n]$, $\cU = \{U_1, \ldots, U_n\}$ be a $\cF$-realization of $\scC$. Then for define $\cV = \{V_1, \ldots, V_n\}$ where 
\[ V_i = U_i \cap \left( \bigcap_{j\in \sigma} U_j\right),
\qquad \text{and}\qquad Y = X \cap \left( \bigcap_{j\in \sigma} U_j\right). \]
Then $\scD = \code(\cV, Y)$. To check this, as above, we can associate each point $p\in Y$ to a codeword by $p\mapsto\{j\in [n] \mid p\in V_j\}$. Since $Y =  X \cap \left( \bigcap_{j\in \sigma} U_j\right)$, each of these codewords will contain $\sigma$, and we will obtain every codeword of $\scC$ containing $\sigma$ in this way. 
\end{proof}

\begin{figure}[h]
  \includegraphics[width=\textwidth]{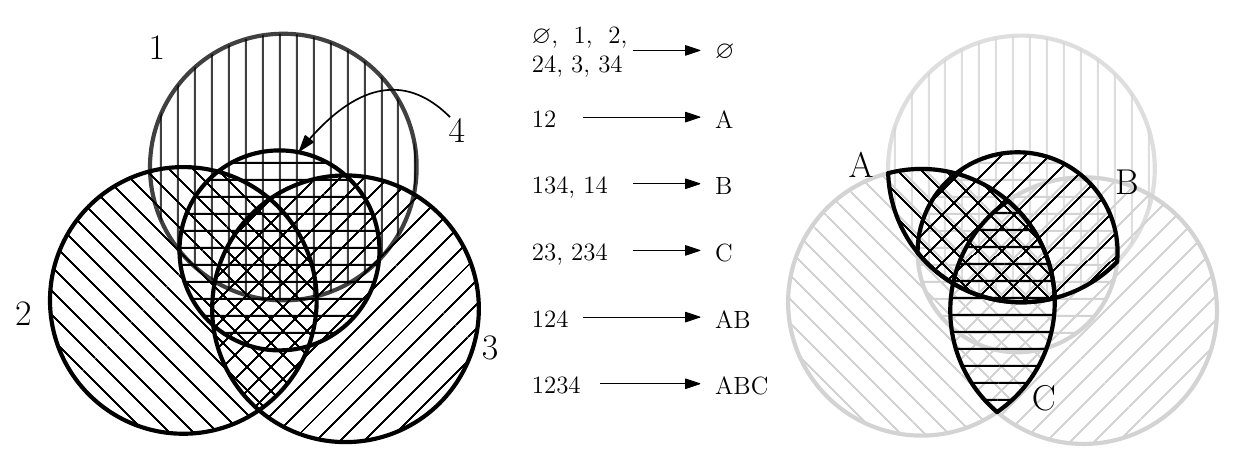}
  \caption{Example of construction from proof of \cref{lem:int_closed}}
\end{figure}

Our first application of \cref{lem:int_closed} is to good cover codes.  A code $\scC$ is a \emph{good cover code} if there exist open sets $U_1, \ldots, U_n$ realizing $\scC$ which form a good cover, i.e. all intersections $\bigcap_{i\in \sigma} U_i$ are either empty or contractible. 
Good cover codes are precisely the codes with no local obstructions, as proved by \cite[Theorem 3.13]{chen2019neural}.  
Codes with local obstructions formed the first known class of non-convex codes \cite{curto2017makes}.  
Recall the \emph{link} of a face $\s$ in a simplicial complex $\D$ is the subcomplex
	\[ \link_\s(\D) = \{ \t\in \D \mid \s\cap\t = \varnothing, \s\cup\t\in\D\}. \] 
For a code $\scC$, $\D(\scC)$ is the simplicial complex of $\scC$, obtained by taking the closure of $\scC$ under taking subsets.
A neural code $\scC$ has a \emph{local obstruction} if there is some $\sigma \in \D(\scC) \sm \scC$ such that $\link_\sigma(\Delta(\scC))$ is not contractible.

We show that codes with no local obstructions form a down-set in $\pcode$. 
The only requirement to be an open set in some good cover is contractibility, and the family of contractible sets is not intersection-closed.
Instead, we consider the sets $U_1, \ldots, U_n$ in one particular good cover and their intersections as our intersection-closed family. 

%

\begin{cor} 
The set of codes with no local obstructions is a down-set in $\pcode$. \label{thm:good_cover}
\end{cor}

\begin{proof}
	Let $\scC$ be a code with no local obstructions. 
	By \cite[Theorem 3.13]{chen2019neural},  $\scC$ is a good cover code.
	Fix a good cover $\cU = \{U_1, \ldots, U_n\}$ realizing $\scC$.
	Let $\cF_\cU$ denote the family of sets obtained by arbitrary intersections of sets in $\cU$, together with the empty set.
	This family still forms a good cover.
	Any code $\scD \leq \scC$ is $\cF_\cU$-realizable by \cref{lem:int_closed}; it is therefore a good cover code and thus has no local obstructions.
\end{proof}

Armed with these results, we look at the codes of oriented matroids and those lying below them. 
In particular, we examine the intersection-closed family of interiors of convex polytopes in $\R^n$.  
We say that a code $\cC$ is polytope convex if there exists a collection of interiors of convex polytopes $\cP = \{P_1, \ldots, P_n\}$ and a bounding convex polytope $X$ such that $\cC = \code(\cP, X)$. 
Note that polytope convex codes are open convex. 
\cref{lem:int_closed} implies that the image of any polytope code under a surjective morphism is also a polytope code. 
Thus, since the codes of representable oriented matroids correspond to  codes of hyperplane arrangements, all codes which lie below a representable oriented matroid are polytope codes.  
We prove the converse, showing that every polytope code is itself the image of the code of an oriented matroid under some surjective morphism. 
This demonstrates that polytope codes are a down-set generated by the set of representable oriented matroid codes. 

We begin by noting that codes below oriented matroids have no local obstructions. This result appears in different language in \cite{edelman2002convex}. We flesh this out. 

\begin{prop} \label{prop:no_local}
  Let $\cM$ be an oriented matroid.
  If $\scC \leq \sfL^+\cM$ in $\pcode$, $\scC$ is a good cover code, and thus has no local obstructions. 
\end{prop}

\begin{proof}
Edelman, Reiner, and Welker define a simplicial complex $\Delta_{\mathrm{acyclic}}(\cM)$ which is identical to  $\Delta(\sfL^+\cM)$ \cite{edelman2002convex}. Proposition 11 and Lemma 13 of  \cite{edelman2002convex}  establish that if $\sigma \in  \Delta(\sfL^+\cM) \setminus \sfL^+\cM$, then $\link_{\sigma}  \Delta(\sfL^+\cM) $ is contractible. Thus,  $\sfL^+\cM$ has no local obstructions, and is therefore a good cover code. 
By \cref{thm:good_cover}, good cover codes form a down-set in $\pcode$, so if $\scC \leq \sfL^+\cM$ in $\pcode$, then  $\scC$ has no local obstructions. 
\end{proof}

\setcounter{introtheorem}{1}
\begin{introtheorem}\label{thm:polytope_matroid}
  A code $\scC$ is polytope convex if and only if there exists a
  representable oriented matroid $\cM$ such that 
  $\scC \leq\sfL^+ (\cM)$. 
\end{introtheorem}

\begin{proof}


	($\Rightarrow$)
	A polytope is an intersection of half-spaces, so this follows from \cref{prop:covectors_match_cover} and \cref{lem:int_closed}.
	
	($\Leftarrow$)
	Let $\scC$ be a polytope convex code with $(\cV,X)$ a realization of $\scC$ with convex polytopes $V_i$ and bounding convex set $X$.
	Without loss of generality, we can choose $X$ to be a convex polytope.
	Then each $V_i$ is the intersection of a collection of open half spaces $U_{i1}, \ldots, U_{ik_i}$, and $X$ is the intersection of open half spaces $X_1, \ldots, X_k$.
	Now, let $\scH= \code(\{U_{11}, \ldots, U_{1k_1}, \ldots, U_{nk_n}, X_1, \ldots, X_k\}, \R^d)$. Let $\scH'$ be the trunk of the neurons associated to $X_1, \ldots, X_k$.
	Now, we define a surjective morphism $f: \scH' \to \scC$ as follows.
	Choose trunks $T_1, \ldots, T_n$ of $\scH'$ by $T_i = \tk_{\scH'}(\{i1,\dots,ik_i\})$.
	Let $f$ be the morphism defined by the trunks $T_1, \ldots, T_n$.
	We now show that its image is $\scC$.  
	
	To do this, construct the realization of $f(\scH')$ given in the proof of \cref{lem:int_closed}. This construction gives the realization
	\begin{align*}
	V_j' = \bigcap_{i = 1}^{i = k_j}U_{ji}
	\end{align*}
	relative to the convex set $X = \bigcap_{i= 1}^k X_i$. 
	Thus, $f(\scH') = \code(\{V_1, \ldots, V_n\}, X) = \scC$. 
\end{proof}


In dimension two, every convex code admits a realization with convex polytopes \cite{bukh2022planar}.
Thus, in dimension two, we can strengthen Theorem \ref{thm:polytope_matroid} as follows: 

\begin{cor}
  A code $\scC$ including the empty codeword is open convex
  with minimal embedding
  dimension two if and only if there exists a representable
  affine oriented matroid $(\cM, g)$ of rank three such that
  $\scC \leq \tk_{\sfL^+(\cM)}(g)$. 
\end{cor}

\begin{proof}
  First suppose $\scC$ is a planar convex code. By intersecting
  each of the convex sets with the same sufficiently large ball, we obtain
  a realization of $\scC$ by bounded convex sets.
Then by Theorem 1 of \cite{bukh2022planar}, $\scC$ has a
realization with interiors of convex polygons in $\R^{2}$.
These convex polygons are intersections of half-spaces.
Let  $(\cM, g)$  be the affine oriented matroid of the
corresponding affine, oriented hyperplane arrangement.
Notice that $\cM$ is representable and rank three,
since it arises from the centralization of a hyperplane
arrangement in $\R^2$. The covectors in the affine space of
$(\cM, g)$ each contain $g$ in their positive part. 
Then by the argument used in the proof of Theorem
\ref{thm:polytope_matroid},  $\scC \leq \tk_{\sfL^+(\cM)}(g)$. 
 
Now suppose that $\scC \leq \tk_{\sfL^+(\cM)}(g)$,
where $(\cM, g)$ is a representable affine oriented matroid of rank two. 
Then by Theorem \ref{thm:polytope_matroid},
$\scC$ has a realization with intersections of half spaces in $\R^2$,
and is thus convex with minimal embedding dimension two. 
\end{proof}



\section{Non-convex codes} 
\label{sec:non_convex}
In dimensions three or higher, it is unknown whether every convex code has a realization with convex polytopes. However, the contrapositive to Theorem~\ref{thm:polytope_matroid} still helps us characterize non-convex codes. 
If $\scC$ is not convex, one of two possibilities hold: either $\scC$ does not lie below any oriented matroid,
or $\scC$ lies below only non-representable oriented matroids in $\pcode$.
In this section, we prove that codes with local obstructions as well as
``sunflower codes''
do not lie below {\em any} oriented matroids. We also construct a new class of
non-convex codes which lie below non-representable oriented matroids.

%

\subsection{Sunflower codes do not lie below oriented matroids}

The first example of a non-convex code with no local obstructions,
\[ \scC = \{\emptyset, 123, 13, 134, 14, 145, 23, 2345, 3, 34, 4, 45\}, \]
appeared in \cite{lienkaemper2017obstructions}. In \cite{jeffs2019morphisms}, Jeffs uses this code to construct a smaller non-convex code $\scC_2 \leq \scC$ with no local obstructions,
	\[ \scC_2 = \{\emptyset, 1236, 13, 135, 23,  234, 4, 456, 5, 6  \}. \]
This code  is minimally non-convex, in the sense that any code $\scC' < \scC_2$ in $\pcode$ is convex. The proofs that $\scC$ and $\scC_2$ are not convex depend on
the $n=3$ case of the following theorem:

%
%
%
%

\begin{thm}[\cite{jeffs2019sunflowers}, Theorem 1.1] \label{thm:sunflower} Let $U_1, \ldots, U_n$ be convex open sets in $\R^{n-1}$ such that for all $i, j\in [n]$, $U_i \cap U_j = \bigcap_{k\in [n]} U_k$. Then any hyperplane which passes through each of $U_1, \ldots, U_n$  passes through $\bigcap_{k\in [n]} U_k$. 
\end{thm}

Jeffs uses this theorem to construct an infinite family $\{\scC_n\}_{n\geq 2}$ of minimally non-convex codes with no local obstructions generalizing $\scC_2$; we refer to these as ``sunflower codes.'' In the rest of this subsection, we define the code $\scC_n$ for $n\geq 2$ and give a proof that for all $n\geq 2$, the code $\scC_n$ does not lie below any oriented matroid, representable or otherwise. 

\begin{defn}[\cite{jeffs2019sunflowers}, Definition 4.1]
  Let $n \geq 2$, $P = \{p_1,\ldots,p_{n+1}\}$ and $S = \{s_1,\ldots,s_{n+1}\}$ be sets of size $n+1$.
Denote by $\scC_n \subseteq 2^{P\cup S}$ the code that consists of the following codewords: 
\begin{itemize}
	\item $\varnothing$; 
	\item $S \cup \{p_{n+1}\}$; 
	\item $P$;
	\item the codeword $X \cup \{s_{n+1}\}$  for each $\varnothing \subsetneq X \subsetneq \{s_1,\dots,s_n\}$;
	\item the codewords $\{p_i\}$ for each $1 \leq i \leq n+1$;
	\item and $(S \setminus \{s_i\}) \cup \{p_i\}$ for each $1 \leq i \leq n$. 
  \end{itemize}
\end{defn}

We will refer to the regions indexed by $P$ as {\em petals},
and the regions indexed by $S$ as {\em simplices}.
\begin{figure} [ht!]
	\includegraphics[width = 4 in]{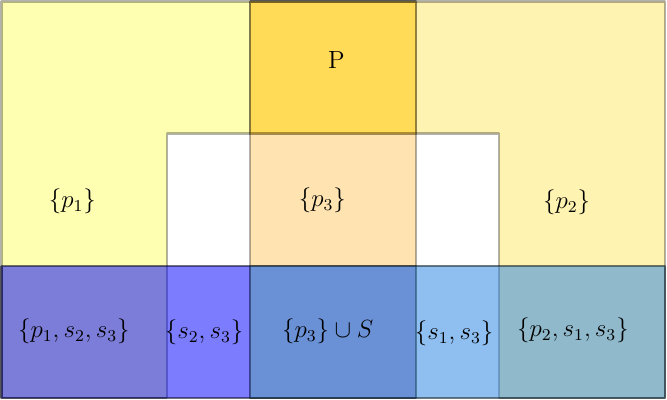}
	\caption{A good cover realization of $\scC_2~=~\{\emptyset, 23, 13, 4, 5, 6, 234, 135,  1236, 456\}.$ Here $P = \{1,2,3\}$ and $S = \{4,5,6\}$.} 
\end{figure}

The proof of  \cref{thm:bad_sunflower} depends on some basic facts about tope graphs of oriented matroids. The \emph{tope graph} $\cT$ of an oriented matroid $\cM$ is a graph whose vertices are the topes of $\cM$, and whose edges connect pairs of topes which differ by one sign. A subgraph $\mathcal Q \subseteq \cT$ is called $T$-convex if it contains the shortest path between any two of its members. Any $e\in E$ divides the tope graph into two \emph{half-spaces} $\cT_e^+ = \{W\in \cW \mid  e\in W^+\}$ and $\cT_e^- = \{W\in \cW \mid  e\in W^-\}$.
A subgraph $\cQ \subseteq \cT$ is $T$-convex if and only if it is an intersection of half-spaces \cite[Proposition 4.2.6]{bjorner1999oriented}.

\begin{introtheorem}
	For each $n \geq 2$, the code $\scC_n \not \leq \sfL^+\cM$ for any oriented matroid $\cM$.
\end{introtheorem}

\begin{proof}
Fix $n \geq 2$.
Suppose to the contrary that there is an oriented matroid $\cM$ such that $\scC_n \leq \sfL^+\cM$.
For ease of notation, let $\scM$ denote the code $\sfL^+\cM$.
Since $\emptyset \in \scC_n$, we can assume without loss of generality that  $\scC_n = f(\scM)$  for some code morphism $f$.

Denote the ground set of $\cM$ by $E$. The map $f$ must be defined by
trunks
	 \[\tk_{\scM}(\pi_1), \dots, \tk_{\scM}(\pi_{n+1}), \tk_{\scM}(\sigma_{1}), \ldots,\tk_{\scM}(\sigma_{n+1}),\]
with $\pi_i,\sigma_i \subseteq E$ corresponding to $p_i$ and $s_i$ respectively. 

\vspace{3mm}
\noindent  {\bf Claim 1:} There is a tope $T$ of $\cM$ such that
    $\left(\bigcup_{i=1}^{n+1} \s_i\right)\cup \left(\bigcap_{j=1}^{n} \pi_j\right)\cup \pi_{n+1} \subseteq T^+$.\\
  Roughly speaking, we are producing a codeword in the intersection of the
  last petal and all simplices, which also lies in the convex hull of the other petals.
  
Define a morphism $g:\scM \to 2^{[n+1]}$ by the trunks $T_i = \tk_{\scM}(\tau_i)$, with 
$\tau_i = \sigma_i \cup \left(\,\,\bigcap_{j = 1}^{n}{ \pi_j}\,\right)$ for $i = 1, \ldots, n+1$.
Let $\scD = g(\scM)$.

Since $(S \setminus \{s_i\} )\cup \{p_i\} \in \scC_n$ for each $i \in [n]$, we deduce that $[n+1]\setminus i$ is a codeword of $\scD$ for each $i\in [n]$. Thus, $\link_{\{n+1\}}(\Delta(\scD))$ is either a hollow $(n-1)$-simplex or a solid $(n-1)$- simplex. Since we have defined $\scD$ as the image of an oriented matroid code, it cannot have local obstructions.
The codeword $\{n+1\}$ is not in $\scD$; if it were, then $f(g^{-1}(\{n+1\}))$
would contain a codeword of $\scC$ including $s_{n+1}$ without any other $s_i$. 
No such codeword
exists in $\scC$.
Thus $\link_{\{n+1\}}(\Delta(\scD))$ must be contractible and so
 must be a solid $(n-1)$-simplex; therefore, $[n+1]$ is a codeword of $\scD$. 

Based on the trunks defining $g$, we know that
$\left(\bigcup_{i =1}^{n+1}\sigma_i\right) \cup \left( \bigcap_{j = 1}^{n} \pi_j \right)\subseteq c$ for any codeword $c\in g^{-1}([n+1])$.
By definition of $f$, we must also have $S \subseteq c$ for any codeword $c \in f(g^{-1}([n+1]))$; however, the only codeword of $\scC_n$ which contains $S$ is $S \cup \{p_{n+1}\} $.
Thus, there is a codeword of $\scM$ containing $\left(\bigcup_{i =1}^{n+1}\sigma_i\right)
\cup \left(\bigcap_{j = 1}^{n} \pi_j \right)\cup \pi_{n+1}$. This implies that
$\cM$ has a covector $X$ such that 
 $\left(\bigcup_{i =1}^{n+1}\sigma_i\right)
\cup \left(\bigcap_{j = 1}^{n} \pi_j \right) \cup \pi_{n+1} \subseteq X^+$. To produce a tope satisfying the condition, take $T = X\circ W$ for any tope $W$ of $\cM$.

\vspace{3mm}
  
\noindent  {\bf Claim 2:}
$ \pi_{n+1} \cup \left( \,\,\bigcap_{ j=1}^{n}\,\pi_j\right) \subseteq T^+ $   implies $\,\,\bigcup_{ j=1}^{n+1}\,\pi_j \subseteq T^+ $ for any tope $T$ of $\cM$.\\
The intuition here is that the last petal must intersect the convex
hull of the other petals {\em only} in the common intersection of all petals.

\begin{figure}
\includegraphics[width = 2.5 in]{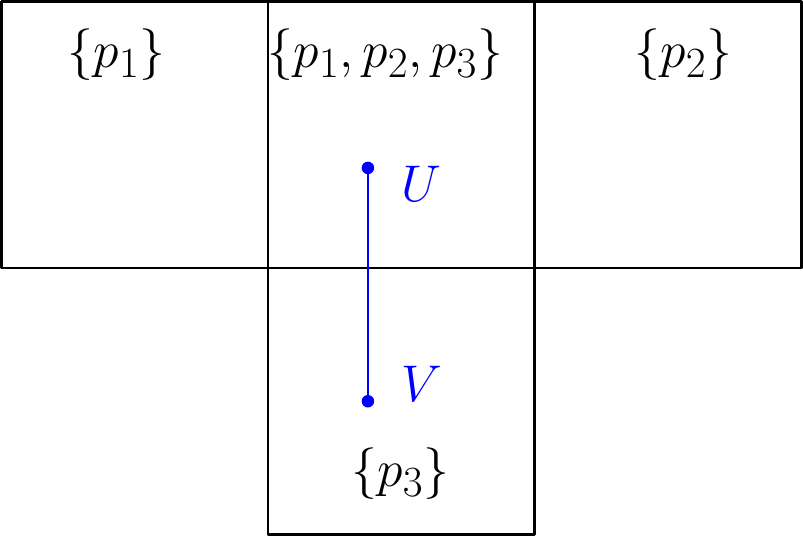}
\caption{Any path from a tope $U$ with $\left(\bigcup_{ j=1}^{n+1}\,\pi_j\right) \subseteq U^+$ to a tope $V$ with $\left(\bigcup_{ j=1}^{n+1}\,\pi_j\right) \not\subseteq V^+$ must cross an edge in  $\left(\bigcap_{ j=1}^{n+1}\,\pi_j\right)$. Analogously, a path from a point in the atom $P = \{p_1, p_2, \ldots, p_n\}$ to the atom $\{p_{n+1}\}$ must cross the boundaries of $p_1, p_2, \ldots, p_{n}$ all at one time. }
\end{figure}

Let $U$ be a tope with $\left(\bigcup_{ j=1}^{n+1}\,\pi_j\right) \subseteq U^+$. Such a tope must exist, since $P \in \scC_n$.   Suppose for the sake of contradiction that there exists a tope $V$ such that 
\[
\pi_{n+1} \cup \left( \,\,\bigcap_{ j=1}^{n}\,\pi_j\right) \subseteq V^+, \text{   but   }\bigcup_{ j=1}^{n+1}\,\pi_j \not \subseteq  V^+.
\]
Consider a shortest path from $U$ to $V$ in the tope graph of $\cM$.  Each edge of the tope graph is naturally labeled by the ground set element $e$ by which the two incident topes differ.  By the $T$-convexity of intersections of half-spaces in the tope graph, each tope along this path has
$\pi_{n+1} \cup \left( \,\bigcap_{ j= 1}^{n}\,\pi_j\right) $ in its positive part, so no edge is labeled with an element of $\bigcap_{ j=1}^{n}\,\pi_j$. 

Thus at some point along the path from $U$ to $V$, we must cross an edge $(T, W)$ labeled by a ground set element 
$e \in
\left(\,\,\bigcup_{ j=1}^{n+1}\,\pi_j \right) \setminus
\left(\,\,\bigcap_{ j=1}^{n+1}\,\pi_j \right)$.
Choose the first such edge  $(T, W)$ labeled with ground set element $e$. By our choice of $e$, there exist
$k, \ell \in [n]$ such that $e\in \pi_k$, and $e\notin \pi_\ell$.
This means $\pi_k \not\subseteq W^+$, whereas $\pi_\ell \subseteq W^+$. Then
$\{p_\ell, p_{n+1}\} \subseteq f(W^+)$, but $f(W^+) \neq P$. However, the only codeword of $\scC_n$ containing $\{p_\ell, p_{n+1}\}$ is $P$, so we have reached a contradiction.
Therefore, no such tope $V$ may exist.

\vspace{6mm}

\noindent By Claim 1, $\cM$ must have a tope $T$ which has
$\left(\bigcup_{i=1}^{n+1} \s_i\right)\cup \left(\bigcap_{j=1}^{n} \pi_j\right)  \cup \pi_{n+1} \subseteq T^+$.
Because $T$ satisfies
$\left(\bigcap_{i=1}^{n+1} \pi_i\right) \cup \pi_{n+1} \subseteq T^+$, Claim 2 implies
that $\bigcup_{i=1}^{n+1} \pi_i \subseteq T^+$.
Therefore,
$\left(\bigcup_{i=1}^{n+1} \pi_i\right) \cup \left(\bigcup_{i=1}^{n+1} \s_i\right) \subseteq T^+$,
but this implies $f(T) = P\cup S \in \scC_n$, a contradiction. 
\end{proof}

By showing that the family of codes $\{\scC_n\}_{n\geq 2}$ do not lie below oriented matroids, we have given an alternate proof that these codes do not have realizations with interiors of convex polytopes.
 This proof is significantly different in structure than the original proof that these codes are not convex using \cref{thm:sunflower}, which is in turn proved by induction on dimension. 
 In contrast, our proof makes no reference to rank or dimension, and does not use induction. 
 While the codes $\{\scC_n\}_{n\geq 2}$ are not open convex, they do have realizations with \emph{closed} convex sets, which can even be chosen to be (non-full dimensional) closed convex polytopes. 
 Notice that  \cref{thm:polytope_matroid}  establishes that if $\scC$ has a  realization with \emph{interiors} of convex polytopes, then  $\scC \leq \sfL^+\cM$. 
 However, the fact that a code has a realization with closed convex polytopes does not guarantee this. 
 Further, in showing that these codes do not lie below any oriented matroids at all, we have established that, even while these codes are good cover codes, their obstructions to convexity are somehow still topological in nature. 

\subsection{Representability and convexity}
	Having exhibited that many well-known non-convex codes do not lie below any oriented matroids at all, we now exhibit a family of non-convex codes which lie below non-representable oriented matroids.
	For each uniform, rank 3 affine oriented matroid $(\cM, g)$, we construct a code $\scC(\cM, g)$ which is convex if and only if $\cM$ is representable (recall a uniform oriented matroid is one in which all circuits have the same cardinality).
	Moreover, this code is always the image of an oriented matroid under a code morphism.

	
	Consider a uniform, affine oriented matroid $(\cM, g)$ of rank 3.
	A \emph{pseudoline} is a simple curve $L$ in $\R^2$, unbounded in both directions, which partitions the plane into unbounded pieces $\R^2 = L^+ \sqcup L \sqcup L^-$.
	By the topological representation theorem (\cite[Section 1.3]{bjorner1999oriented},\cite{folkman1978oriented}),
	$(\cM, g)$ can be represented by a uniform arrangement of piecewise linear pseudolines, that is, a family $\cP = \{L_i\}_{i\in[n]}$ of pseudolines such that each pair intersects exactly once and no more than two meet at any point.
	The sign vectors of this arrangement are the covectors of $(\cM,g)$.
	An example is illustrated in Figure~\ref{F:pseudoline_code}(a).


	Note that the oriented matroid of a pseudoline arrangement is completely determined by the order in which each line meets all of the other lines. We can record this information as follows: 
	Let $L_1, \ldots, L_n$ be a pseudoline arrangement.
	For each pseudoline, fix one end of the pseudoline as the ``head".
	Let $\pi_i(j)$ denote the index $k$ such that $L_j$ is the $k$-th pseudoline we encounter as we follow $L_i$ from the head to the tail. 
	
	We use this order to define a code $\scC(\cM, g)$.
	We then use the concept of order-forcing introduced in \cite{jeffs2020order} to prove that the code $\scC(\cM, g)$ is convex if and only if $\cM$ is representable.
	Order-forcing depends on \emph{feasible walks} in the \emph{codeword containment graph}.

	\begin{defn}\label{def:feasible}
		Let $\scC \subseteq 2^{[n]}$ be a neural code.
		The \emph{codeword containment graph} of $\scC$ is the graph $G_{\scC}$ whose vertices are codewords of $\scC$, with edges $\{\sigma,\tau\}$ when either $\sigma\subsetneq\tau$ or $\tau\subsetneq \sigma$. 
		A $\sigma,\tau$ walk  $\sigma=v_1,v_2,...,v_k=\tau$ in $G_\scC$ is called \emph{feasible} if $v_i\cap v_j\subseteq v_m$ for all  $1\le i<m<j\le k$.
		A sequence of codewords $\sigma_1,...,\sigma_k$ is \emph{order-forced} if every feasible $\sigma_1,\sigma_k$ walk contains that sequence as a subsequence.
	\end{defn}
	
	Order forcing constrains the realizations of a code by forcing certain sequences of codewords to correspond to straight-line paths in all convex realizations. In any realization of a code, the \emph{atom} corresponding to codeword $\s$ is the region $A_\s = \bigcap_{i \in \s}U_i \setminus \bigcup_{j\notin\s} U_j$.
	
	\begin{thm}[\cite{jeffs2020order}, Theorem 1.1]\label{thm:order-forcing}
	Let $\sigma_1, \sigma_2, \ldots, \sigma_k$ be an order-forced sequence of codewords in a code $\scC\subseteq 2^{[n]}$.
	Let $\cU = \{U_1, \ldots ,U_n\}$ be a (closed or open) convex realization of $\scC$, and let $x\in A_{\sigma_1}$ and $y\in A_{\sigma_k}$. Then the line segment $\overline{xy}$ must pass through the atoms of $\sigma_1, \sigma_2, \ldots, \sigma_k$, in this order.
	\end{thm} 
	
	Now, we construct the code $\scC(\cM, g)$
	so that a sequence of codewords along each pseudoline is order-forced. 

	\begin{defn}
		Let $(\cM, g)$ be a uniform, affine oriented matroid of rank 3 with pseudoline arrangement $L_1, \ldots, L_n$.
		Without loss of generality, assume we have labeled pseudolines $L_1, \ldots, L_n$, with their heads in clockwise order around the outside of the plane.
		An example is illustrated in Figure \ref{F:pseudoline_code} (a). \\
		
		$\scC(\cM, g)$ is a code on $n+ 2 + n^2 +2n = (n+1)(n+2)$ neurons, labeled: \\
		
		\begin{center}
		\begin{longtable}{@{} p{0.3\textwidth}  p{0.65\textwidth} @{}} 
			$a_1,\ldots, a_n$: & Strips corresponding to each pseudoline of $\cM$.\\ 
			$b_\ell, b_r$: & Strips corresponding to two new ``boundary'' pseudolines whose positive quadrant includes all pseudoline intersections.\\ 
			$c_{i,j}$ for all $i,j\in[n]$: & $n$ neurons along each $a_i$ to apply order-forcing.\\ 
			$d_{s,i}$ for all $s\in\{\ell, r\}, i\in[n]$: & $n$ neurons along $b_r$ and $b_\ell$ to apply order-forcing.
		\end{longtable}
		\end{center}
		
		\vspace{5mm}
		
		The codewords of $\scC(\cM, g)$  are as follows: \\
		
		\begin{center}
		\begin{longtable}{@{} p{0.3\textwidth} p{0.65\textwidth} @{}} 
			$b_\ell b_r d_{\ell,1}d_{r,1} $: & Intersection of the two boundary strips.\\ 
			$b_s d_{s,j}$: & Order-forcing along each boundary strip \\ & ($s = \ell,r$, $j = 1, \ldots, n$.)\\ 
			$b_r a_id_{r,i}d_{r,i+1} c_{i,1}$: & Intersection of each pseudoline with {\bf right}
			boundary strip, with order-forcing neurons. ($i = 1, \ldots, n-1$)\\ 
			$b_r a_n d_{r,n}c_{n,1}$: & Intersection of final pseudoline with {\bf right} boundary strip (one less
			order-forcing neuron is required.) \\ 
			$b_\ell  a_{n + 1 -i} d_{\ell,i}d_{\ell,i+1} c_{n+1-i,n}$:  & Intersection of each pseudoline with {\bf left}
			boundary strip plus order-forcing neurons. ($i = 1, \ldots, n-1$.) \\ 
			$b_\ell a_1d_{\ell ,n}c_{1,n}$: & Intersection of first pseudoline with {\bf left} boundary strip.  \\ 
			$a_i c_{i,j}$: & Order-forcing along each pseudoline \\ & ($i = 1, \ldots, n$,
                        $j = 1, \ldots, n$) \\ 
			$a_i a_{j} c_{i,\pi_i(j)}c_{i,\pi_i(j)+1}$  & Pairwise intersections of pseudolines plus order-forcing \\ \hfill $c_{j,\pi_j(i)}c_{j,\pi_j(i)+1} $: & ($i = 1, \ldots, n$, $j = 1, \ldots, n-1$.) \\ 
		\end{longtable}
		\end{center}
		
		\vspace{3mm}
		
		We include an example of a good cover realization of this code in Figure~\ref{F:pseudoline_code}(b). 
	\end{defn}
	
	\begin{figure}
\includegraphics[width=0.9\textwidth]{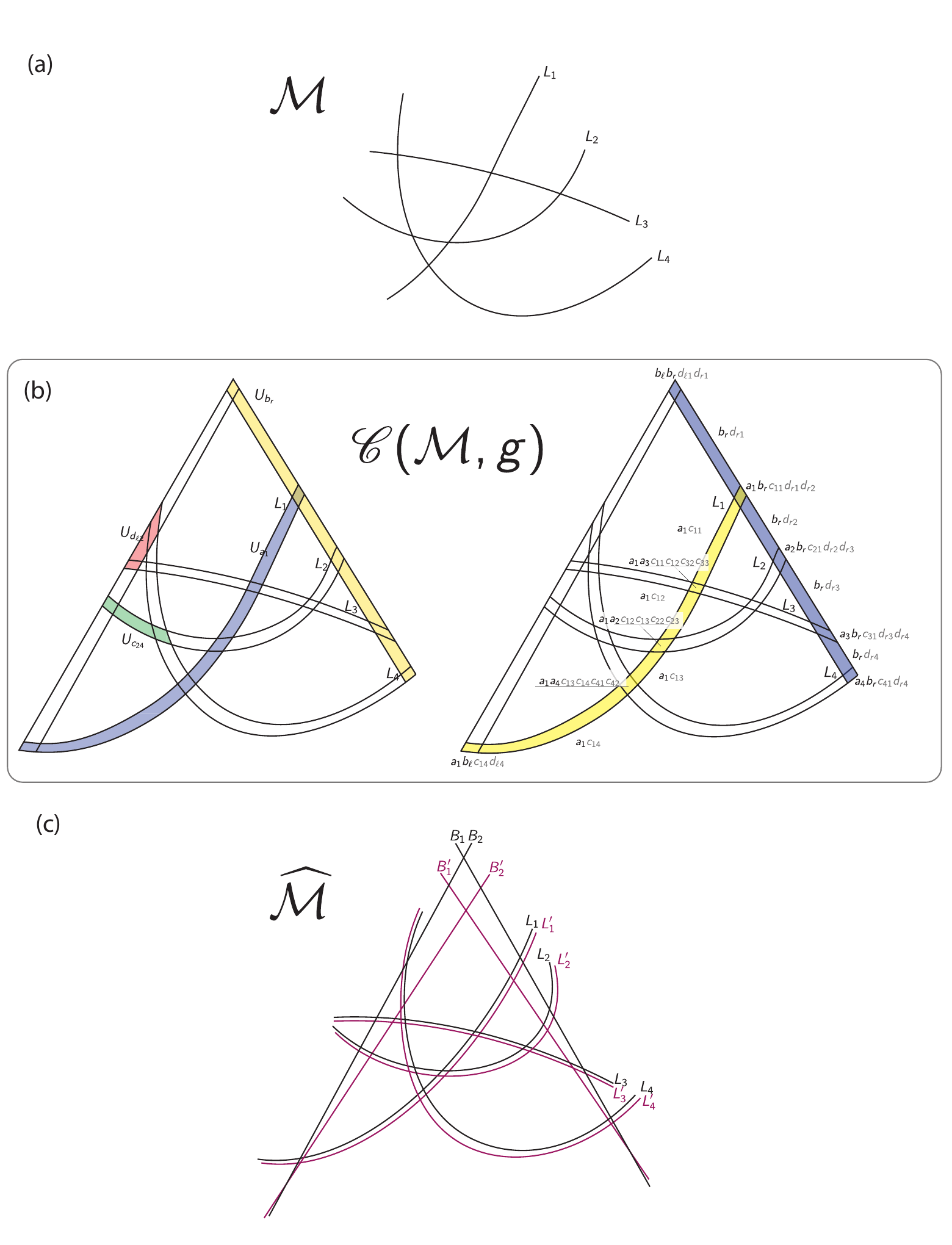}

		\caption{(a) A pseudoline arrangement for the oriented matroid $\cM$. 
		(b) A good-cover realization of the code $\scC(\cM, g)$. Left, representative open sets are shown. Right, select codewords are labeled. 
		(c) A pseudoline arrangement for $\widehat{\cM}$.  }
		\label{F:pseudoline_code}
	\end{figure}

	\begin{prop}\label{prop:img_of_matroid}
	For any uniform, rank 3  affine oriented matroid $(\cM, g)$, there exists a rank 3 oriented matroid $\widehat\cM$ such that $\scC(\cM, g) \leq \sfL^+\widehat\cM$.  
	\end{prop}
	

	\begin{proof}
		We describe the pseudoline arrangement associated to $\widehat\cM$. 
		Fix a piecewise-linear pseudoline arrangement $L_1, \ldots, L_n$ representing $(\cM, g)$ consistent with the labeling in $\scC(\cM, g)$. 
		Let $B_r$ be a line which meets $L_1, L_2, \ldots, L_n$ in the clockwise order consistent with the labeling. 
		Let $B_\ell$ be a line which meets $B_r$ and then $L_n, L_{n-1}, \ldots, L_{1}$ in the opposite of this clockwise order.  
		Orient $B_r$ and $B_\ell$ such that $B_r^+$ and $B_\ell^+$ are the half-spaces containing all bounded cells of the pseudoline arrangement.
		Orient each $L_i$ such that $B_r\cap B_\ell$ lies in $L_i^-$. 
		
		Now, for each $i\in [n]$, we define a pseudoline $L_i'$ which  acts as a translation of $L_i$ into its positive half-space. 
		That is, we let $L_i'$ be a pseudoline which intersects $B_\ell, B_r$,  and each $L_j, j\neq i$ in the same order as $L_i$, and such that for each other pseudoline $L$, the intersections of $L_i$ and $L_i'$ are adjacent along $L$. 
		Further, we ensure that $L_i, L_i'$ do not intersect. 
		Orient $L_i'$ so that $L_i \subseteq L_i'^+$.
		Define $B_\ell', B_r'$ and orient them analogously.
		This pseduoline arrangement is illustrated in Figure~\ref{F:pseudoline_code}(c). 

		
		Finally, we produce an oriented matroid $\widehat \cM$ from this pseudoline arrangement by fixing a ground set element $h$ such that the set of covectors of the pseudoline arrangement is the affine space of $(\widehat \cM, h)$.  We claim the oriented matroid of this arrangement, $\widehat\cM$, lies above $\scC(\cM, g)$.
		The morphism $f$ such that $f(\tk_{\sfL^+\widehat\cM}(h)) = \scC(\cM, g)$ is defined by the the trunks
		\begin{align*}
		\{T_{a_i}\}_{i = 1, \ldots, n} \cup \{T_{b_r}, T_{b_\ell}\} \cup \{T_{c_{i,j}}\}_{i = 1, \ldots, n, j = 1, \ldots, n} \cup \{T_{d_{s,j}}\}_{s=r,\ell, j = 1, \ldots, n}.
		\end{align*}
		corresponding to the neurons of $\scC(\cM, g)$.
		We let $i, i'$ (for each $i\in[n]$), $r, r', \ell,$ and $ \ell'$ be the ground set elements corresponding to $L_i, L_i', B_r, B_r', B_{\ell},$ and $B_{\ell}'$ respectively. 
		These trunks are defined as follows:
		\begin{align*}
		T_{a_i} &:= \tk(\{i, i', r, \ell\}) \mbox{ for  } i = 1, \ldots, n\\
		T_{b_r} &:= \tk(\{r, r', \ell, n'\})\\
		T_{b_\ell} &:= \tk(\{\ell, \ell', r, 1'\})\\
		T_{d_{r,1}} &:= \tk(\{r, r', \ell, 1'\})\\
		T_{d_{r,i}} & := \tk(\{r, r', {i-1}, {i}'\}) \mbox{ for } i = 2, \ldots, n\\
		T_{d_{\ell,1}} &:= \tk(\{\ell, \ell', r, n'\})\\
		T_{d_{\ell,i}} &:= \tk(\{\ell, \ell', {n-i+2}, {n-i+1}'\}) \mbox{ for } i = 2, \ldots, n.
		\end{align*}
		
		In order to define $T_{c_{i,j}}$, we introduce some notation. Let 
		\begin{align*}
		\mathrm{right}(i, j)  = 
		\begin{cases}
		 j \mbox{ if } j < i\\
		 j'  \mbox{ if } j > i
		\end{cases} \qquad
		\mathrm{left}(i, j)  = 
		\begin{cases}
		 j' \mbox{ if } j < i\\
		 j  \mbox{ if } j > i\\
		\end{cases}.
		\end{align*}
		That is, $\mathrm{right}(i, j)$ is whichever of $j, j'$ the pseudoline $L_i$ meets first as we follow it from its intersection with $B_r$ to its intersection with $B_\ell$, and $\mathrm{left}(i,j)$ is whichever it hits second.
		Now, we define
		\begin{align*}
		T_{ c_{i,1}} &:= \tk(\{i, i', r,\mathrm{left}(i, \pi_i(1))\}) \mbox{ for } i = 1, \ldots, n.\\ 
		T_{ c_{i,j}} &:= \tk(\{i, i',  \mathrm{right}(i, \pi_i(j-1)),  \mathrm{left}(i, \pi_i(j))\}) \mbox{ for } i = 1, \ldots, n,\: j = 2, \ldots, n-1.\\ 
		T_{ c_{i,n}} &:= \tk(\{i, i', \mathrm{right}(i, \pi_i(n-1)), \ell \}) \mbox{ for } i = 1, \ldots, n,\: j = 2, \ldots, n-1.\\ 
		\end{align*}
		Finally, we verify that the map $f(\sigma) = \{ s \mid \sigma\in T_s\}$ has image $\scC(\cM, g)$.
		Each $T_s$ specifies the open set $U_s$ given by the
                intersection of the open half-spaces indexed
                by $s$. By construction, the $\{U_s\}$ give rise to a
                good cover realization of $\scC(\cM, g)$.
                
	\end{proof}

	\begin{introtheorem}
	 Let $\cM = (E, \cL)$ be a uniform, rank 3 oriented matroid. Then for $g\in E$, the code $\scC(\cM, g)$ is convex if and only if $\cM$ is representable. \label{prop:cvx_iff_rep}
	 \end{introtheorem}
	 
	\begin{proof}
	First, we show that if $\cM$ is representable, $\scC(\cM, g)$ is convex. Note that by Proposition \ref{prop:img_of_matroid}, we have that $\scC(\cM, g) \leq \sfL^+\widehat\cM$. Also note that by construction, if $\cM$ is representable, then so is $\widehat\cM$. Therefore, by Theorem \ref{thm:polytope_matroid}, if $\cM$ is representable, $\scC(\cM, g)$ is convex.

Next, we show that if $\scC(\cM, g)$ is convex, then $\cM$ is representable. Note that the following sequences are order-forced in $\scC(\cM, g)$. 
 
\begin{enumerate}

 \item The only feasible path from $  b_r   b_\ell   d_{r,1} d_{\ell,1}$ to  $ b_{r} a _n d_{r,n}$ in $G_{\scC(\cM, g)}$ is 
 \begin{align*}  b_r   b_\ell   d_{r,1} d_{\ell 1} \lra     b_r   d_{r,1} \lra   b_r   a _1  d_{r,1} d_{r,2} \lra   b_r   d_{r,2} \lra \cdots \lra  b_{r} d_{r,n} \lra  b_{r} a _n d_{r,n}
 \end{align*}
 
  \item The only feasible path from $  b_r   b_\ell   d_{r,1} d_{\ell,1}$ to  $  b_\ell  a _1 d_{\ell,n}$ in $G_{\scC(\cM, g)}$ is 
  \begin{align*}
    b_r   b_\ell   d_{r,1} d_{\ell, 1} \lra     b_\ell   d_{\ell, 1} \lra   b_\ell   a _n  d_{\ell, 1} d_{\ell, 2} \lra   b_r   d_{\ell, 2} \lra \cdots \lra   b_\ell  d_{\ell, n} \lra   b_\ell  a _1 d_{\ell, n}
  \end{align*}

  \item For each $i$, the only feasible path from $ b_r  a_i  c_{i,1}  d_{r,i} d_{r,i+1}$ to $ a _i   b_\ell   c_{i,n} d_{\ell, (n-i+1)}$ in  $G_{\scC(\cM, g)}$ is 
  \begin{align*} 
    b_r  a _i  c_{i,1}  d_{r,i} d_{r,i+1} 
    &\lra  a_i c_{i,1}
    \lra  a_i  a_{\pi\inv_i(1)}  c_{i,1} c_{i,2}c_{\pi\inv_i(1), \pi_{\pi\inv_i(1)}(i)}c_{\pi\inv_i(1), \pi_{\pi\inv_i(1)}(i)+1} 
    \lra  a _i  c_{i,2} \lra  \\
   \cdots &
   \lra  a _i  c_{i,(n-1)} 
   \lra  a_i  a_{\pi\inv_i(n-1)}  c_{i,(n-1)} c_{i,n}c_{\pi\inv_i((n-1)), \pi_{\pi\inv_i(n-1)}(i)}c_{\pi\inv_i(1), \pi_{\pi\inv_i(n-1)}(i)+1}\lra  \\ 
   \cdots& \lra  a _i  c_{i,n}
   \lra  a _i   b_\ell   c_{i,n} d_{\ell, (n-i+1)}
  \end{align*}
 \end{enumerate}

We claim if $\scC(\cM, g)$ is convex, then it has a realization in the plane.
Suppose that  $\scC(\cM, g)$ is convex, and fix a realization $\cU$ in $\R^d$.  
Choose points $p_1, p_2, p_3$ in the atoms $A_{  b_r   b_\ell   d_{r,1} d_{\ell, 1}}$, $A_{ b_{r} a _n d_{r,n}}$, and $A_{  b_\ell  a _1 d_{\ell, n}}$ respectively. 
We will show that each atom in this realization has a nonempty intersection with $\conv(p_1, p_2, p_3)$. 
By order forcing (1), the line  from $p_1$ to $p_2$ must pass through the atoms of all codewords containing $b_r$ in the listed order. 
Likewise, by order forcing (2), the line from $p_1$ to $p_3$ must pass through the atoms of 
all codewords containing $b_\ell$ in the listed order. 

 
In particular, we have shown that for each $i$, the atoms of $\s = b_ra _id_{r,i}d_{r,(i+1)}c_{i,1}$ and $\t = b_\ell  a _i d_{\ell, (n+1-i)} d_{\ell, (n+2-i)} c_{i,1}$ have a nonempty intersection with $\conv(p_1, p_2, p_3)$. For each $i$, pick a point $q_i \in \conv(p_1, p_2, p_3) \cap A_{\s}$ and a point $r_i \in \conv(p_1, p_2, p_3) \cap A_{\t}$. Applying order forcing (3) for each $i$, we have that the line from $r_i$ to $q_i$ passes through the atoms of  
all codewords containing $a_i$, in the listed order. 
This accounts for every codeword of $\scC(\cM, g)$. Thus, intersecting the open  sets in $\cU$ with the plane $\aff(p_1, p_2, p_3)$ produces a two-dimensional convex realization of $\scC(\cM, g)$. 

Now, we obtain a straight line arrangement for $(\cM, g)$ in this plane by extending the line segment from $q_i$ to $r_i$ to be a line.
By uniformity of $\cM$ and our choice of bounding lines, every pair of pseudolines intersects in $B_\ell^+ \cap B_r^+$ and so no new intersections are introduced.
Notice that by order forcing (3), this line meets the sets $U_{ a _1}, \ldots, U_{ a _n}$ in the order consistent with the pseudoline arrangement. Thus, if this code is convex, 
$\cM$ is representable. 
\end{proof}

Proposition \ref{prop:cvx_iff_rep} demonstrates that matroid representability and
convex code realizability are intertwined. One consequence is that
non-representable oriented matroids are a new source
for constructing non-realizable codes:

\begin{cor}
  There is an infinite family of non-convex codes which lie below
  oriented matroids in $\pcode$.
 \end{cor}

\begin{proof}
There are infinitely many non-representable uniform oriented matroids of rank 3 \cite[Proposition 8.3.1]{bjorner1999oriented}. 
By Proposition~\ref{prop:cvx_iff_rep}, $\scC(\cM, g)$ is non-convex for each of these.
By Proposition \ref{prop:img_of_matroid}, $\scC(\cM, g) \leq \sfL^+\widehat\cM$. 
\end{proof}

\begin{ex} Let $(\cM, g)$ be the uniform non-Pappus matroid from \cite{shor1991stretchability}, whose pseudoline arrangement appears in Figure~\ref{fig:non-papp}. This matroid is non-representable, since a realization of it would violate Pappus's hexagon theorem. Then $\scC(\cM, g)$ is a non-convex code with no local obstructions. 

\begin{figure}[ht!]
\includegraphics[width = 3.5 in]{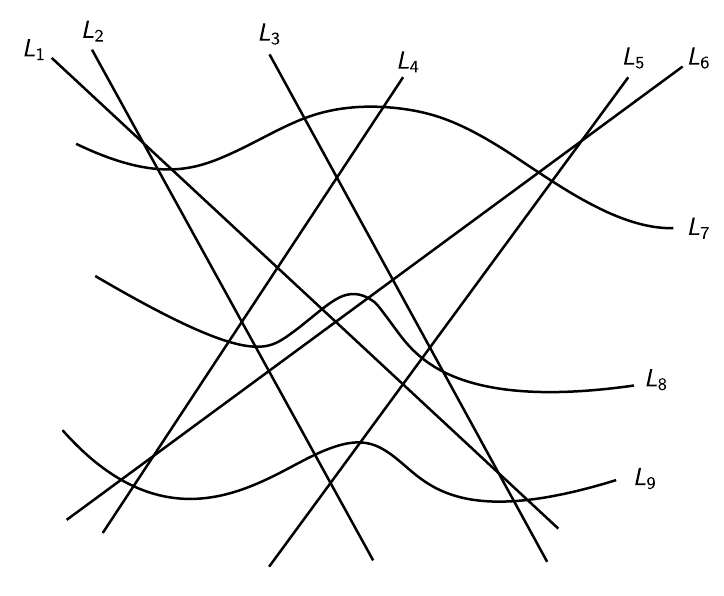}
\caption{The pseudoline arrangement for the uniform non-Pappus matroid. The pairs of lines 23, 45, 78, 79, and 89 are taken to intersect outside of the figure. }
\label{fig:non-papp}. 
\end{figure}

\end{ex}

\subsection{The convex code decision problem is NP-hard}

	We now turn to the computational aspects of convex codes.
	Using the relationship between convex codes and representable oriented matroids (Theorem~\ref{thm:witness}), we demonstrate the convex code decision problem is NP-hard and $\exists\R$-hard, though it remains open whether the convex code decision problem lies in either of these classes, or is even decidable.
	The complexity class $\exists \R$, read as \emph{the existential theory of the reals}, is the class of decision problems 
	of the form 
		\[ \exists (x_1\in \R) \ldots \exists(x_n\in \R) P(x_1, \ldots, x_n), \]
	where $P$ is a quantifier-free formula whose atomic formulas are polynomial equations, inequations, and inequalities in the $x_i$.  In other words, a problem in $\exists \R$ defines a semialgebraic set over the real numbers and asks whether or not it contains any points \cite{broglia2011lectures}. Many well known problems in computational geometry lie in $\exists \R$, including some problems very similar to determining whether a code is convex. For instance, determining whether a graph is the intersection graph of convex sets in the plane is $\exists \R$ complete \cite{schaefer2009complexity}. 
	
Theorem~\ref{thm:witness} implies the convex code decision problem is at least as difficult as deciding if an oriented matroid is representable.
This decision problem is $\exists\R$-complete \cite{mnev1988universality, shor1991stretchability, sturmfels1987decidability, richter1999universality}: given any decision problem in $\exists \R$, there is a polynomial time algorithm to produce an oriented matroid (presented in terms of covectors) which is realizable if and only if the decision problem has a positive answer. 
 Therefore the convex code decision problem is $\exists\R$-hard.
	
\begin{introtheorem}
	Any problem in $\exists \R$ can be reduced in polynomial time to the problem of determining whether a neural code is convex. 
\end{introtheorem}
	
\begin{proof}

By the Mn\"ev-Sturmfels universality theorem (see \cite{mnev1988universality, sturmfels1987decidability, shor1991stretchability, bjorner1999oriented,  richter1999universality}), determining whether a rank 3 uniform oriented matroid (presented in terms of covectors) is representable is complete for the existential theory of the reals.
By Proposition~\ref{prop:cvx_iff_rep}, a rank 3 uniform oriented matroid $\cM$ is representable if and only if $\scC(\cM, g)$ is a convex neural code. Further, the number of neurons in $\scC(\cM, g)$  is quadratic in the size of the ground set of $\cM$, and the number of codewords of $\scC(\cM, g)$ is less than the number of covectors of $\cM$. Finally, we can construct $\scC(\cM, g)$  from the covectors of $\cM$ in polynomial time. Any problem in $\exists\R$ can be reduced in polynomial time to deciding representability of a uniform oriented matroid and thus convexity of the corresponding code.
\end{proof}

\noindent 
Since any $\exists \R$ complete problem is also $\mathrm{NP}$-hard, we have as a corollary that determining whether a code is convex is $\mathrm{NP}$-hard.
 \begin{cor}
The problem of determining whether a code is convex is $\mathrm{NP}$-hard, where the problem size is measured in the number of codewords. 
\end{cor}
Notice that because we can perform this reduction of a problem in $\exists \R$ to a neural code in polynomial time, this result holds even when we measure the problem size in terms of the number of codewords, which may be exponentially large in the number of neurons.  Again, this NP hardness result is not surprising. For instance, it parallels the result that recognizing whether a simplicial complex is the nerve of convex sets in $\R^d$ is $\mathrm{NP}$-hard for $d\geq 2$ \cite{tancer2010d}.

\section{Categories of codes, matroids, and rings}\label{sec:algebra}

\subsection{The Neural Ring}\label{S:neuralring}
	To set the stage for the functorial connections between combinatorial codes and oriented matroids, we begin with a brief discussion of the functor $\sfR:\Code\to\NRing$ defined in \cite{jeffs2019morphisms}, and its relation to the combinatorial relations of a code, introduced in \cite{curto2013neural}.
	Recall that the neural ring of a code $\scC$ is $R_\scC = \F_2[x_1,\dots,x_n] / I_\scC$, where $I_\scC$ is the vanishing ideal of $\scC$ as a variety in $\F_2^n$.
	This is the ring of $\F_2$-valued functions on $\scC$ with distinguished coordinate functions $x_1,\dots,x_n$, that is, $x_i(\s) = 1$ iff $i \in \s$.
	The category $\NRing$ is the category of neural rings together with \emph{monomial maps}, ring homomorphisms $\phi:R_\scD \to R_\scC$ which map the coordinate functions of $R_\scD$ either to products of coordinate functions in $R_\scC$ or to $0$.
	By restricting to this class of homomorphisms, the functor $\sfR$ which takes a code to its neural ring is a contravariant equivalence of categories \cite[Theorem 1.6]{jeffs2019morphisms}.
	For $f: \scC \to \scD$ a morphism of codes defined by trunks $T_i = \tk_\scC(\s_i)$ for $i \in [m]$, the ring homomorphism $\sfR f: R_\scD \to R_\scC$ sends the coordinate function $x_i$ in $R_\scD$ to the product $x^{\s_i}$ in $R_\scC$.
	
	The pseudo-monomials in $I_\scC$ provide a dual description of $\scC$.
	They record the dependencies among the elements of $[n]$, or, equivalently, among the sets $U_i$ in any realization of $\scC$.
	If $\scC = \code(\cU,X)$, then \cite[Lemma 4.2]{curto2013neural} implies:
		\begin{align}
			x^\s(1-x)^\t \in I_\scC \iff \bigcap_{i\in\s} U_i \subseteq \bigcup_{j\in\t} U_j.\label{eq:comborelation}
		\end{align}
	Containment relationships as in the right hand side of \eqref{eq:comborelation} are called the \emph{combinatorial relations} of $\scC$.
	As a generating set for $I_\scC$, the minimal pseudo-monomials, i.e.\ the minimal combinatorial relations, are sufficient to recover the code $\scC$.
	The minimal proper pseudo-monomials in $I_\scC$ are called the \emph{canonical form} of $\scC$ \cite{curto2013neural}.
	The following lemma shows that the structure of a pseudo-monomial ideal encodes the weak elimination axiom (axiom (C\ref{axiom:weakelimC})) of oriented matroid circuits.

	\begin{lem}\label{L:CFcircuitaxioms}
		Let $\scC = 2^{[n]}$ be a combinatorial code.
		Denoting
		pseudo-monomials $x^\s(1-x)^\t$
		as sets $\s \cup \bar\tau \subseteq \pm[n]$, the minimal relations of $\scC$ satisfy circuit axiom (C\ref{axiom:weakelimC}) (weak elimination). 
	\end{lem}
	
	\begin{proof}
		
		Suppose $p_1 = x^\s(1-x)^\t$ and $p_2 = x^\a(1-x)^\b$ are minimal in $I_\scC$, with $e \in \s\cap\b$.
		Then
			\begin{align*}
				x^{\a\setminus\s}(1-x)^{\b\setminus(\t\cup e)} p_1 + x^{\s\setminus(\a\cup e)}(1-x)^{\t\setminus\b} p_2 = x^{\sigma \cup \alpha \setminus e}(1-x)^{\tau \cup \beta \setminus e} \in I_\scC.
			\end{align*}
		Thus, some minimal pseudo-monomial $x^{Z^+}(1-x)^{Z^-}$ in $I_\scC$ divides $x^{\sigma \cup \alpha \setminus e}(1-x)^{\tau \cup \beta \setminus e}$, i.e.\ $Z^+ \subseteq (\s\cup\a) \setminus e$ and $Z^- \subseteq (\t\cup\b) \setminus e$, which is exactly circuit axiom (C\ref{axiom:weakelimC}).
	\end{proof}
	
	Note that, while the \emph{proper} circuits of an oriented matroid satisfy axiom (C\ref{axiom:weakelimC}), we must include improper pseudo-monomials of the form $x_i(1-x_i)$ in order for the generators of $I_\scC$ to satisfy  (C\ref{axiom:weakelimC}).  While elements of the canonical form are \emph{minimal} combinatorial relations, they do not satisfy axiom (C\ref{axiom:incomparableC}) (incomparability). Combinatorial relations on the same support need not be equal or opposite: for instance, the combinatorial relations of the code $\scC = \{\varnothing, 1, 2, 3, 123\}$ are $U_1\cap U_2 \subseteq U_3, U_2\cap U_3 \subseteq U_1$, and $U_1 \cap U_3 \subseteq U_2$, which are all supported on the set $\{1,2,3\}$.

	The relationship between pseudo-monomials in $I_\scC$ and codewords in $\scC$ is analogous to the relationship between circuits and topes.
	In light of \cref{L:CFcircuitaxioms}, the oriented matroid analogue of $\sfR$ maps an oriented matroid $\cM$ to an ideal generated by the circuits of $\cM$ and then the depolarization map $\sfD$ is simply the algebraic analogue of $\sfW^+$.
	As we will see, most of the work involved in establishing these connections is in showing $\sfW^+$ and $\sfS$ are functors.

\subsection{Oriented matroids to neural codes}\label{sec:functorW}
We now show that the map $\sfW^+$ is a contravariant functor from the category $\OMatroid$ whose objects are acyclic oriented matroids and whose morphisms are strong maps, to the category $\Code$ whose objects are neural codes and whose objects are code morphisms.

We define strong maps in terms of convexity following \cite{hochstattler1999linear}. First, we include the requisite information on convexity for oriented matroids.

\begin{defn}
	A subset $S \subseteq \pm E$ is \emph{convex} in $\cM$ if for all $x \notin S$, there is no circuit $C \in \cC(\cM)$ such that $-x \in C \subseteq S \cup \{-x\}$.
	The \emph{convex closure} of a set $S \subseteq \pm E$ is the intersection of all convex sets that contain $S$.
\end{defn}

\begin{rmk} This definition differs from \cite[Exercise 3.9, p 152]{bjorner1999oriented}
  in that it acts on subsets of $\pm E$ rather than $E$. The intuition behind this definition
comes from vector arrangements.
A signed-linear dependence on $S\cup \{-x\}$ gives a signed-linear
representation of $x$ in terms of $S$: just add $x$ to both sides of the equation.
This can be rescaled to a convex combination of elements in $S$; therefore, $x$ should be in the 
convex closure.
\end{rmk}

We now define strong maps:

	\begin{defn}
		Let $\cM_1, \cM_2$ be a pair of oriented matroids on ground sets $E_1, E_2$ and $\uf: E_1\cup \{\lop\}\to E_2\cup \{\lop\}$ such that $\uf(\lop) = \lop$. Extend $\uf$ to a map $f$ on the signed ground sets by $f(-e) = -f(e)$, where $\lop = -\lop$. We say that $f$ induces a \emph{strong map} $\phi_f: \cM_1\to \cM_2$ if whenever
	$S \subseteq \pm E_2$ is a convex set of $\cM_2$, 
$f^{-1}(S) \subseteq \pm E_1$ is a convex set of $\cM_1$. 

	\end{defn}

        \begin{rmk}
          We briefly explain the loops in the definition of strong maps. We want the function on
          sets to be well-defined while still allowing some elements to ``disappear,'' so we add
          $\{\circ\}$ in the target to absorb the disappearing elements. Since strong maps
          between matroids on the same ground set have certain duality properties,
          we include $\{\circ\}$ in the source as well.
          \end{rmk}

The following lemma gives us an equivalent definition of convexity in terms of topes.
We will make use of a corollary (\cref{cor:topeconvex}) along the way to proving $\sfW^+$ is a functor.
	
	\begin{lem}\label{prop:topeconvex}
		A subset $S \subseteq \pm E$ is convex if and only if for all $x \notin S$ and
		$A \subseteq S$ containing no signed circuits, there exists a tope $X \in \cW(\cM)$ such that
		$A \cup \{-x\} \subseteq X$.
		\end{lem}
	\begin{proof}
		Assume that for all $x \notin S$ and $A \subseteq S$ containing no signed circuits, there is a tope $X$ with $A \cup \{-x\} \subseteq X$.

		Suppose that $S$ is not convex, by way of contradiction. Then there exists some $x \notin S$
		for which there is a circuit $C$ with $-x \in C \subseteq S \cup \{-x\}$. But,
		$A = C\setminus \{-x\}$ is a subset of $S$ containing no signed circuits (by axiom (C\ref{axiom:incomparableC})), and if any tope contained
		$A \cup \{-x\}$, that would contradict tope-circuit orthogonality.

		For the reverse implication, we prove the contrapositive using the four-painting axioms \cite[Theorem 3.4.4 (4P)]{bjorner1999oriented}. Suppose that there is some set $A \subseteq S$
	 containing no signed circuits and an element $x \notin S$ such that $A \cup \{-x\}$ is
		not contained in any tope. Paint
		the ground set to be black and white coincident with $A \cup \{-x\}$, and to be red on the
		remaining elements. By the four-painting axioms, there must be a circuit
		supported on the elements of $A \cup \{-x\}$; this proves that $S$ is not convex.
\end{proof}

	\begin{cor}\label{cor:topeconvex}
		Every tope of a loopless matroid is convex.
	\end{cor}
	\begin{proof}
		Let $X$ be a tope. By tope-circuit orthogonality,
		there is no circuit contained in $X$. Consider $x \notin X$.
		Since topes have full support, $x \notin X$ implies $-x \in X$. This means that for
		any $A \subseteq X$, the set $A \cup \{-x\} \subseteq X$, which is a tope. Therefore $X$ is convex.
	\end{proof}

	Now we define the contravariant functor $\sfW^+: \OMatroid \to \Code$.
	We restate the map on objects and add the action on morphisms.

\begin{defn}

	Let $\cM$ be an acyclic oriented matroid. 
	Take $\sfW^+ \cM$ to be the code consisting of the positive parts of topes of $\cM$,
	\[ \sfW^+ \cM = \{ W^+ \mid W \in \cW(\cM)\} \subseteq 2^E. \]

	Let $\phi_f: \cM_1\to \cM_2$ be a strong map with associated set map
 $\underline{f}:E_1\cup \{\lop\} \to E_2 \cup \{\lop\}$.
	Then, take $\sfW^+ \phi_f: \sfW^+ \cM_2 \to \sfW^+ \cM_1$, to be the map on codewords
	$(\sfW^+ \phi_f)(\sigma) = \underline{f}^{-1}(\sigma)$.

	\end{defn}


In order to prove that $\sfW^+$ is a functor, we must prove that $\sfW^+ \phi$ is actually
a well-defined function with the desired domain.
At this point, acyclicity becomes necessary.

\begin{ex}\label{ex:acyclicnecessary}
	Consider the matroid $\cM_1$ on ground set $E = \{1,2,3\}$ defined by the columns
	of the matrix \[
	\left[ \begin{array}{cccc}
		1 & -1 & 0 \\
		0 & 0 & 1 \\
		\end{array}
		\right]
	\]
	The topes of $\cM_1$ are $ \{1\bar 2 3,1 \bar 2 \bar 3, \bar 1 23,\bar 1 2\bar 3\}$.
	Let $\cM_2$ be the rank-1 matroid on one element obtained by contracting
	the first two columns of $\cM_1$.  That is, $\cM_2$ is the oriented matroid on ground set $[1]$ with topes $\{\bar 1, 1\}$. The contraction is the strong map induced by the set map $\uf(1) =\uf(2) = \lop$, $\uf(3) = 1$. 

	Passing to $\Code$, we have $\sfW^+\cM_2 = \{ \varnothing, 1\}$ and
	$\sfW^+ \cM_1 = \{1, 13,2, 23\}$.
	For the functor to work, we would need $\sfW^+\phi(1) = 3$ to be the positive part
	of some tope, but there is no such tope. By demanding that the matroids are
	acyclic we avoid this problem. Acyclic oriented matroids are also loopless,
        so topes of acyclic oriented matroids have full support.
\end{ex}

\begin{prop}
	Let $\cM_1$ and $\cM_2$ be acyclic oriented matroids on $E_1$ and $E_2$ respectively,
	and $\phi_f: \cM_1 \to \cM_2$ a strong map induced by $\uf : E_1\cup \circ \to E_2\cup\circ$. 
	If $X \in \cW(\cM_2)$ is a tope, there is a tope $Z \in \cW(\cM_1)$ such that
	$\uf^{-1}(X^+) = Z^+$.

	\label{prop:welldefined}
	\end{prop}

\begin{proof}
	Since both matroids are loopless, topes of each have full support on their ground sets.
	By \cref{cor:topeconvex}, $X \cup \{\lop\}$ is convex, and since $f$ is a strong map, we
	conclude that $f^{-1}(X \cup \{\lop\}) = \uf^{-1}(X^+)^+ \sqcup \uf^{-1}(X^-)^- \sqcup \pm \uf^{-1}(\lop)$
	is convex.
	We claim that omitting the positive-signed elements of $f^{-1}(\lop)$ to obtain
	$Z := \uf^{-1}(X^+)^+ \sqcup \uf^{-1}(X^- \cup \{\lop\})^-$ retains convexity.

	If not, then by nonconvexity of $Z$, there is $x \notin Z$,
		such that $-x$ is
        in some circuit $C$ contained in $Z \cup \{-x\}$. Thus,
        $x$ would be an element in the convex closure of $Z$ but not in $Z$ itself,
        rendering $Z$ nonconvex. 
	Because $Z$ has full support, this means $-x \in Z $, implying $ C \subseteq Z$. Because
	$\cM_1$ is acyclic, $C$ must have at least one element $x \in \uf^{-1}(X^+)^+$.
        But this implies
	that $-x$ should be in the convex closure of $f^{-1}(X \cup \{\lop\})$,
        contradicting convexity.

	Finally, by \cref{cor:topeconvex}, we note that a maximal signed convex set must be a tope,
	indicating that $Z$ is a tope satisfying our constraints.
	\end{proof}

Thus, \cref{prop:welldefined} confirms that the map of codes has
the desired domain, so it is well-defined as a map of sets. We need to confirm that
this set map is also a morphism of codes (i.e.\ the preimage of a trunk is a trunk).

\begin{prop}\label{P:strongmaptomorphism}
	For a strong map of acyclic oriented matroids	$\phi_f:\cM_1 \to \cM_2$,
	the map of neural codes given by $\sfW^+ \phi_f(\sigma) = \uf^{-1}(\sigma)$ is a morphism of
	codes.
\end{prop}

\begin{proof}
Let $\scC_i = \sfW^+\cM_i$ for $i = 1,2$.
It is sufficient to check that the preimage of a simple trunk (i.e.\ the trunk of a singleton set) is a trunk. Thus, we compute $(\sfW^+\phi_f)\inv\tk_{\scC_1}(i)$. Let $\tau\in	(\sfW^+\phi_f)\inv\tk_{\scC_1}(i)$, so $(\sfW^+\phi_f)(\tau) \in \tk_{\scC_1}(i)$. By our definition of $\sfW^+\phi_f$, this is equivalent to the condition $f\inv(\tau)\in \tk_{\scC_1}(i)$. By the definition of a trunk, this is equivalent to $i\in f\inv(\tau)$, or $f(i)\in \tau$. Thus, $\tau\in \tk_{\scC_2}(f(i))$ if and only if $\tau\in	 (\sfW^+\phi_f)\inv\tk_{\scC_1}(i)$. Therefore
	\[ (\sfW^+\phi_f)\inv\tk_{\scC_1}(i) = \tk_{\scC_2}(f(i)). \]
Thus, the map $\sfW^+\phi_f$ is a morphism of neural codes. 
\end{proof}

To finish off the proof that $\sfW^+$ is a functor, we need only check that
it respects the identity morphism and composition of morphisms.

\begin{prop}
	The identity strong map on a matroid $\operatorname{id}:\cM \to \cM$ yields
	$\sfW^+ \operatorname{id} : \sfW^+ \cM \to \sfW^+ \cM$ the identity on the corresponding code.

	Given two strong maps $\phi : \cM_1 \to \cM_2$ and $\psi: \cM_2 \to \cM_3$,
	the morphisms $\sfW^+(\psi \circ \phi)$ and $\sfW^+\phi \circ \sfW^+\psi$ from
	$\sfW^+ \cM_3 \to \sfW^+ \cM_1$ are equal.
\end{prop}

\begin{proof}\label{prop:morphism}
	Based on \cref{prop:welldefined}, the map of codes is well-defined.
The composition of strong maps is defined by	$\phi_g \circ \phi_f:=\phi_{g\circ f}$. Then 
	\[ \sfW^+(\phi_g \circ \phi_f)(\sigma) = \sfW^+(\phi_{g\circ f})(\sigma)= (g\circ f)\inv(\sigma) = f\inv \circ g\inv(\sigma) = (\sfW^+\phi_f\circ	\sfW^+\phi_g)(\sigma). \]
	Thus $\sfW^+$ respects composition of morphisms. Next, we check that
	$$\sfW^+(\phi_{\id})(\sigma)= \id \inv(\sigma) = \sigma,$$
	thus $\sfW^+$ respects the identity morphism. 
	Therefore, $\sfW^+$ is a functor. 
	\end{proof}

\begin{prop}\label{P:faithfulnotfull}
	The map $\sfW^+$ is a faithful, but not full, contravariant functor from the category	$\OM$ of acyclic oriented matroids with strong maps to the category $\Code$ of neural codes with code morphisms. 
\end{prop}

Since we have already proven that the map of categories $\sfW^+$ is indeed a functor, we only need
to prove that the functor is faithful but not full to complete the proof of \cref{P:faithfulnotfull}. 

\begin{proof}

	For a given strong map $\phi : \cM_1 \to \cM_2$, it is easy to read out 
	the map on ground sets $E_1 \to E_2$ from the values of $\sfW^+ \phi$. Because
	the set map uniquely determines the strong map, the functor is faithful -- that is, it is injective on morphisms. 
	
	To show that not all morphisms of codes arise from strong maps of oriented matroids we produce
	the following example:

	Take the morphism $f:\scC\to \scD$, where
	\[
\begin{array}{lcl}
  \scC & = & \{\varnothing, 1, 12345, 1235, 1245, 13, 1345, 135, 14, 145, 2, 23, 2345, 235,
  24, 245, 3, 4\} \\
\scD & = & \{\varnothing, 1', 1'2', 2'\},
\end{array}
\]
and the morphism is defined by trunks $\tk_\scC(135)$ and $\tk_\scC(245)$. See \cref{fig:non-strong} for realizations of these codes. By construction, $f$ is a morphism of neural codes. Both codes are hyperplane codes, thus they arise from oriented matroids $\cM_2$ and $\cM_1$. However, the map $f$ does not arise from any strong map. To see this, notice that the proof of \cref{P:strongmaptomorphism} actually proves that the preimage of a simple trunk is a simple trunk for any morphism arising from a strong map. However, by construction, $f\inv(\tk_\scD(1')) =\tk_\scC(135)$, which is not a simple trunk. 

	This proves that the functor is not full.
	\end{proof}
	
\begin{figure}
	\includegraphics[width = 5.5 in]{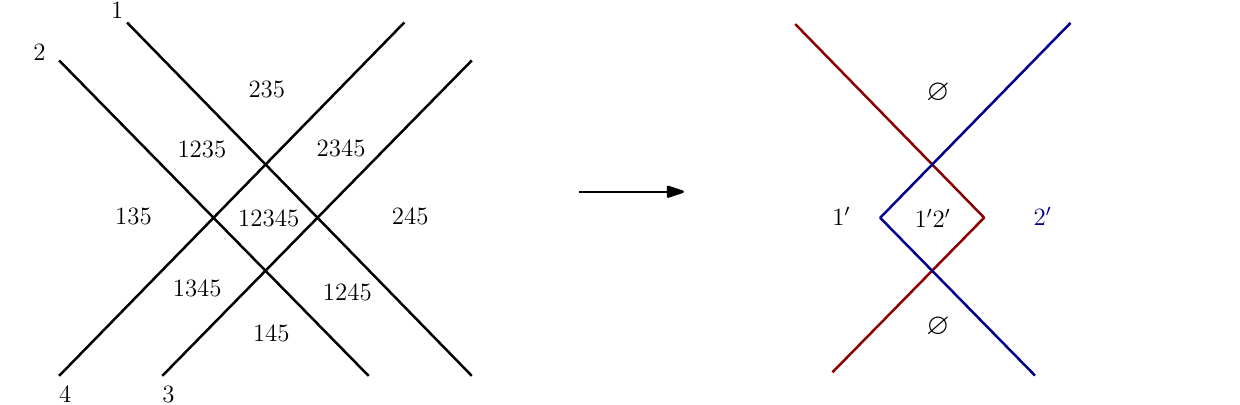}
	\caption{Partial realization of code~$\scC$ (left) and a realization of code~$\scD$ (right).
	To construct the complete realization of $\scC$,
		embed this figure in the plane $z = 1$ in $\R^3$. For $i = 1, 2, 3, 4$, let the plane $H_i$ be the plane spanned by the line embedded in the plane $z = 1$ and the origin, and let the $H_5$ be the plane $z = 0$, oriented up. Notice that the canonical construction of a realization of $\scD$ from the realization of $\scC$ is not a hyperplane realization, even though $\scD$ happens to be a hyperplane code.} 
		\label{fig:non-strong}
\end{figure}

\subsection{Oriented matroids to rings}\label{sec:functorI}
We now describe the oriented matroid ring and show that the map taking an oriented matroid to its associated ring is a functor.
The key ingredient for doing this is the oriented matroid ideal introduced
in \cite{novik2002syzygies}. As defined in that paper, the oriented matroid ideal
is associated to {\em affine} oriented matroids; in other words, oriented
matroids with a distinguished element. We alter their definition to avoid
the need for a distinguished element,
and show that the affine oriented matroid ideal can be constructed algebraically from the oriented matroid ideal.
Finally, we define the oriented matroid ring as the quotient by the Alexander dual ideal.
We define the functor $\sfS$ which takes an oriented matroid to its oriented matroid ring and describe its image,
which we take as our category $\OMRing$.

Fix a field $k$.
We will consider polynomial rings over $k$ with variables indexed by the ground set $E$ of an oriented matroid;
when the indexing set is apparent, we will denote these as $k[\bx,\by]$ or $k[\bx]$.
The affine oriented matroid ideal is defined in \cite{novik2002syzygies} (under
the name ``oriented matroid ideal'') with an equivalent description from their
Proposition~2.8 as follows:

\begin{defn}
  Let $\cM = (E,\cL,g)$ be an affine oriented matroid with $E = \{1,\ldots,n,g\}$.  \label{defn:affOMI}
  
  \noindent ({\bf Covectors})  For every sign vector $Z \in \{0,+,-\}^E$, associate a monomial
  \[ m_{xy}^{(g)}(Z)= \left( \prod_{i : Z_i = +} x_i \right) \left( \prod_{i : Z_i = -} y_i\right)
  \text{  where  } x_g = y_g = 1.\]
  The affine oriented matroid ideal $O_g(\cM)$ is the ideal in
  $k[{\bf x},{\bf y}]$ generated by all
  monomials corresponding to covectors $Z \in \cL^+ = \{X\in \cL \mid X_g = +\}$.\\[-3mm]

  \noindent ({\bf Circuits})  The minimal prime decomposition of the affine oriented matroid ideal is \\ $O_g(\cM) = \bigcap_{C} P_C^{(g)}$,
  where $P_C^{(g)}$ is the ideal generated by variables $\langle x_i,y_j\mid i \in C^+, j \in C^-$, $j \neq g \rangle$,
  and the intersection is over all circuits $C$ such that $g \in C^-$.

\end{defn}

We define the oriented matroid ideal in terms of generators, and show that it also has
this dual description in terms of minimal primes.

\begin{defn}
   \label{def:OMI} 
  Let $\cM$ be a loopless oriented matroid.   
  Let the {\em oriented matroid ideal} $O(\cM)$ denote the ideal generated as \[\left\langle m_{xy}(Z): Z \text{  tope of  } \cM\right\rangle, \text{   where  } m_{xy}(Z)= \left( \prod_{i : Z_i = +} x_i \right) \left( \prod_{i : Z_i = -} y_i\right).\]
  \\[-5mm]  \end{defn}

\begin{rmk}\label{R:OMgens}
  This definition can be extended in a straightforward way to matroids with loops,
  but ``topes'' would be replaced with ``complements of covectors.'' Note that the
  minimal complements of covectors in loopless matroids are indeed the topes.
\end{rmk}

\begin{prop}\label{DP:omi} 
  Let $\cM = (E,\cL)$ be a loopless oriented matroid.
  The minimal prime decomposition of $O(\cM)$ is given by
  $O(\cM) = \bigcap_{C} P_C$,
  where $P_C$ is the ideal generated by variables $\langle x_i,y_j \mid i \in C^+, j \in C^-\rangle$,
  and the intersection is over \emph{all} (proper and improper) circuits $C$.
\end{prop}

\begin{proof}
  Note that $O(\cM)$ is a monomial ideal, as is the intersection of the monomial ideals $\{P_C\}_C$.
  Therefore, it is sufficient to check that the sets of monomials in the two ideals are identical.

  First, consider $m_Z = m_{xy}(\pm [n] \setminus Z)$ for $Z \in \cW(\cM)$. We will show
  that $m_Z \in P_C$ for all $C \in \cC(\cM)$.
  For each element $b \in E$, exactly one of $b$ or $-b$ is in every tope
  $Z$, so $m_Z \in \langle x_b, y_b\rangle$; this covers improper circuits of the form $\{ b,\bar{b}\}$.
  For every proper circuit $C$,  both $\sep(Z,C)$ and $\sep(Z,-C)$ are non-empty by tope-circuit orthogonality.
In this case, there exists $i \in C$ (resp, $-i \in C$) such that
  $i \notin Z$ ($-i \notin Z$); this means that $x_i \mid m_Z$
  for $i \in C$ which implies $m_Z \in P_C$. Since $m_Z \in P_C$ for all types of circuits,
  it is also in the intersection.
  
  In the reverse direction, we show that for any monomial $m$ in $\bigcap_C P_C$, there
  is a tope $Z$ such that $m_{xy}(\pm [n] \setminus Z) \mid m$.
  For all 
  elements $j \in E$, either $x_j \mid m$ or $y_j \mid m$;
  so there exist disjoint sets $I,J$ such that $[n] = I \cup J$ 
  and
  \[m_{I,J} = \left( \prod_{i \in I} x_i \prod_{j \in J} y_j \right)\mid m.\]

  We claim that $Z = I \cup \bar{J}$ is a tope of $\cM$. It is enough to show that
  every circuit $C \in \cC(\cM)$ is orthogonal to $Z$. 
  The fact that $m_{I,J} \in P_C$ and
  $m_{I,J} \in P_{-C}$ means that both
  $\sep(Z,C)$ and $\sep(Z,-C)$ are nonempty, implying orthogonality. 
\end{proof}

The affine oriented matroid ideal can be obtained from the oriented matroid ideal using the following construction.
\begin{prop}
  The affine oriented matroid ideal $O_g(\cM)$
  can be obtained via the following ideal quotient and specialization
  \[O_g(\cM) = [O(\cM): O(\cM \setminus g)]\:
  \vline_{\:\:\substack{x_g = 1 \\ y_g = 0}} \subseteq k[x_1,\ldots,x_n,y_1,\ldots,y_n].\] 
\end{prop}

\begin{proof}
  By \cref{defn:affOMI},
  \[
  \begin{matrix}
    O(\cM)   = \displaystyle \bigcap_{C \in \cC(\cM)} P_C 
    &  = & 
    \displaystyle\bigcap_{\substack{C_g = 0\\ C \in \cC(\cM)}} P_C & \cap &
    \displaystyle\bigcap_{\substack{C_g = +\\ C \in \cC(\cM)}} P_C  & \cap &
    \displaystyle\bigcap_{\substack{C_g = -\\ C \in \cC(\cM)}} P_C \\[8mm]
    &  = &   \displaystyle\bigcap_{\substack{C_g = 0\\ C \in \cC(\cM)}} P_C & \cap &
    \left(\langle x_g \rangle + \displaystyle\bigcap_{\substack{C_g = +\\ C \in \cC(\cM)}}
    P_C^{(-g)}\right)   &
    \cap & \left(\langle y_g \rangle + \displaystyle\bigcap_{\substack{C_g = -\\ C \in \cC(\cM)}} P_C^{(g)}\right)\\[11mm]
    & = &  O(\cM \setminus g) & \cap &
   ( \langle x_g \rangle + O_{-g}(\cM) ) &\cap &  ( \langle y_g \rangle + O_g(\cM) ) 
    \end{matrix}
\]
  
Ideal quotients commute with intersection, so we can apply the
quotient to each component.
The first component becomes the ideal 
quotient of $ O(\cM \setminus g)$ by itself, which is the full ring.
After specializing $x_g = 1$, the second component is also the full ring.
Turning to the third component,
we need to prove that
$\big((\langle y_g \rangle + O_g(\cM)):O(\cM\setminus g) \big)= \langle y_g \rangle + O_g(\cM)$.

A monomial $m$ is in the quotient if and only if
for all $A = \pm [n] \setminus B$ where $B \in \cL(\cM\setminus g)$, either
$y_g \mid m \cdot m_{xy}(A)$ or
there exists covector $Z\in \cL(\cM)$ with $Z_g = -$ such that
$m_{xy}^{(g)}(Z) \mid m \cdot m_{xy}(A)$.
Suppose $ m \cdot m_{xy}(A) \in \langle y_g \rangle $.
Then $y_g \mid m$ since $m_{xy}(A)$ is defined on the deletion by $g$;
this implies that $m \in \langle y_g \rangle$.
Suppose instead that $m \cdot m_{xy}(A) \in O_{g}(\cM)$.
This implies that there is a covector $Z$ of $\cM$  with $Z_g = -$ such that
$m_{xy}^{(g)}(Z) \mid m \cdot m_{xy}(A)$.
By \cite[Prop 3.8.2 (b)]{bjorner1999oriented}, this implies that the support of $m$
is a covector of the matroid. Since $y_g \nmid m_{xy}(A)$, the support of $m$ must include $y_g$,
  implying that its support is a covector $B$ of $\cM$ with $B_g = -$. This implies  $m \in O_{g}(\cM)$. We conclude that $\big((\langle y_g \rangle + O_g(\cM)):O(\cM\setminus g) \big)= \langle y_g \rangle + O_g(\cM)$. Specializing $y_g = 0$ leaves us with $O_g(\cM)$.  
\end{proof}

One more step is needed to make the functor $\sfS$ work. 
The oriented matroid ideal $O(\cM)$ is a square-free monomial ideal; we take its Alexander dual (see e.g.\ \cite[Definition 1.35]{miller2004combinatorial}) to obtain $O(\cM)^\star$.
This takes the oriented matroid ideal and swaps the role of topes and circuits; i.e.\ irreducible components now correspond to topes, and monomial generators to circuits.
Let $W$ be a tope and let $\mathfrak{p}(W) = \langle x_e \mid W_e = +\rangle + \langle y_e \mid W_e = -\rangle$.
Then, for acyclic oriented matroids,
	\begin{align}
		O(\cM)^\star = \langle m_{xy}(C) \mid C \in \cC(\cM)\rangle = \bigcap_{W \in \cW(\cM)} \mathfrak{p}(W). \label{eq:OMstar}
	\end{align}
The oriented matroid ring is then the quotient ring $k[\bx,\by]/O(\cM)^\star$.

\begin{prop}\label{P:Ifunctor}
	Let $\OMatroid$ be defined as above.

	Let $\cM$ be an oriented matroid and $\phi:\cM_1\to\cM_2$ be a strong map of matroids with associated set map $f:E_1\cup \{\lop\}\to E_2\cup \{\lop\}$.

	Define $\sfS \cM = S_\cM = k[\bx,\by]/O(\cM)^\star$.
	Define $\sfS \phi_f: \sfS\cM_1 \to \sfS \cM_2$ by
	\[ (\sfS\phi_f) ( x_i) = 
	\begin{cases} 0 & f(i) =\lop\\
	x_{f(i)} & else. \end{cases}  \hspace{1cm} (\sfS\phi_f) (y_i) = \begin{cases} 0 & f(i) =\lop\\
	y_{f(i)} & else. 
	\end{cases} \]
	We refer to the ring $S_\cM$ as an \emph{oriented matroid ring} and the map $\sfS\phi_f$ as a \emph{strong monomial map}. 
	Then, $\sfS$ is a covariant functor from $\OM$ to $\Ring$. 
\end{prop}

\begin{proof}
We need to prove that this map defines a ring homomorphism, respects the identity morphism, and respects 
composition of morphisms.

We begin by checking that the map $\sfS \phi_f$ is a ring homomorphism. 
Since it is defined as a map on generators, $\sfS\phi_f $ defines a ring homomorphism  
$k[x_1, \ldots, x_{n_1}, y_1, \ldots, y_{n_1}] \to  k[x_1, \ldots, x_{n_2}, y_1, \ldots, y_{n_2}].$
We need to check that this map respects the quotient structure. 
 That is, we must show that if $m\in O(\cM_1)^\star$, then  $\sfS\phi_f(m)\in O(\cM_2)^\star$. 
 
 Since $O(\cM_1)^\star$ is a monomial ideal, it is sufficient to check this for monomials $m = \prod_{i\in I}x_i \prod_{j\in J}y_j.$
If $f(i) = \lop$ (or $f(j)=\lop$) for some $i \in I$ ($j\in J$), then $\sfS \phi_f(m) = 0\in O(\cM_2)^\star$.
Next, we consider the case when $\sfS\phi_f(x_i) = x_{f(i)}, \sfS\phi_f(y_j)= y_{f(j)}$ for all $i\in I, j\in J$.  
Because $O(\cM_1)^\star$ is a monomial ideal, $m \in O(\cM_1)^\star$ implies
  that $\bx^{C+}\by^{C-}$ divides $m$ for some generator $\bx^{C+}\by^{C-} \in O(\cM_1)^\star$. 

 If $f(C)$ is not a signed set, then it contains an improper circuit of the form $\{i, \bar i\}$, so $\sfS\phi_f (m)$ is divided by $x_i y_i$, so $\sfS\phi_f (m)\in  O(\cM_2)^\star$ as desired.
Thus, suppose that $f(C)$ is a signed set. 
We show that $f(C)$ contains a circuit. 
Let $e\in C$. 
We will show that $-f(e)$ is in the convex closure of $f(C)$; this implies that there is a circuit $D$ of $\cM_2$ such that 
$$f(e) \in D \subseteq f(C) \cup \{f(e)\} = f(C), $$ which is what we need.  
Suppose that $-f(e)$ is not in the convex closure of $f(C)$. Then there is some convex set $S$ such that $-f(e) \notin S$, $f(C)\subset S$. 
By the definition of a strong map, $f\inv (S)$ must be convex. 
However, $-e \notin f^{-1}(S)$, and 
$$e \in C \subset f\inv(S) \cup \{e\},$$
contradicting convexity of $S$. 
We conclude that $-f(e)$ is in the convex hull of $f(C)$. Thus, there is a circuit  $D$ of $\cM_2$ such that 
$f(e) \in D \subseteq f(C)$, so $x^{D^+}y^{D^-}$ divides $f(m)$.


To see that $\sfS$ respects the identity morphism, note that if $f(i) = i$ for each $i\in E$, then $\sfS \phi_f(x_i) = x_i$ and $\sfS \phi_f(y_i) = y_i$, so $\sfS\phi_f$ is the identity on $\sfS\cM$. Now, let $\phi_f$ and $\phi_g$ be strong maps. First, suppose neither $f(i) = \lop$ nor $g(f(i)) = \lop$.  Without loss of generality, we check that composition of morphisms is respected on the $x_i$. Then $\sfS(\phi_f \phi_g)(x_i) = x_{f\circ g(i)} = (\sfS f)(\sfS g) x_i$. Now, if either $f(i) = \lop$ or $g(f(i)) = \lop$, then $ \sfS(\phi_f \phi_g)(x_i)  = 0 =  (\sfS f)(\sfS g) x_i$. Thus the map $\sfS$ respects composition of morphisms. 
\end{proof}

We define the category $\OMRing$ to be the category whose objects are oriented matroid rings $S_\cM$ with distinguished generators $x_1, \ldots, x_n, y_1, \ldots, y_n$.  The morphisms of $\OMRing$ are the strong monomial maps $\sfS \phi_f$, where $\phi_f$ is a strong map of oriented matroids.


\subsection{Oriented matroid rings to neural rings and back}\label{sec:functorD}

The final piece of the puzzle is describe the relationship between $\OMRing$
and the category of neural rings $\NRing$. Note that neural rings are defined over $\F_2$, thus we take all rings in this section to be over $\F_2$. 
The vanishing ideal of a code is a \emph{pseudo-monomial} ideal, meaning it has a pseudo-monomial generating set.
Polarization of a pseudo-monomial ideal, introduced in \cite{gunturkun2017polarization},
produces a true monomial ideal which encodes the same combinatorial information.
As $\sfW^+$ is not a full functor and $\sfR$ is an equivalence of categories, there is no reason to expect polarization to be a functor.
Instead, we will use the operation of depolarization to define the functor
$\sfD$ so that $\sfR \circ \sfW^+ = \sfD \circ \sfS$, i.e. the diagram below commutes.
\begin{equation}
\begin{tikzcd}
\OMatroid  \arrow[r, "\sfS"] \arrow[d, "\sfW^+"]
& \OMRing \arrow[d, "\sfD"] \\
\Code \arrow[r,  "\sfR"]
& \NRing
\end{tikzcd} \label{eq:functordiagram}
\end{equation}

\begin{defn}
	Let $S_\cM$ be an oriented matroid ring. 
	Define $\sfD S_\cM$ to be the ring $S_\cM/\langle x_i + y_i -1 \mid i\in [n]\rangle$ with distinguished
	coordinate functions $x_1, \ldots, x_n$. 
	
	If $\phi: S_{\cM_1} \to S_{\cM_2}$ is a morphism in $\OMRing$ 
	with underlying set map $f: E_1 \cup \{\lop\} \to E_2 \cup \{\lop\}$,
	then define $\sfD \phi : \sfD S_{\cM_1} \to  \sfD S_{\cM_2}$  to be the map sending $x_i \mapsto x_{f(i)}$ if $f(i) \neq \lop$ and $x_i \mapsto 0$ otherwise. 
\end{defn}

\begin{prop}\label{prop:D_functor}
	The map $\sfD$ is a functor $\OMRing$ to $\NRing$.
\end{prop}

\begin{proof}
	We first show $\sfD$ maps an oriented matroid ring to a neural ring.
	Denote $S = \F_2[\bx,\by]$ and $D = \langle x_i + y_i - 1 \mid i \in [n] \rangle \subseteq S$.
	Let $\bar{D}$ denote the ideal with the same generators as $D$, but considered as an ideal of $S_\cM$, i.e.\ $\sfD S_\cM = S_\cM / \bar{D}$, and let $S' = O(\cM)^\star + D \subseteq S$.
We apply standard isomorphism theorems to conclude
		\[ \sfD S_\cM = S_\cM / \bar{D} \cong S/S' \cong (S/D) / (S'/D). \]
	Observe that $S/D \cong \F_2[\bx]$ under the map $y_i \mapsto 1-x_i$. Under this same map, $x^\s y^\t \mapsto x^\s(1-x)^\t$, so $S'/D$ is a pseudo-monomial ideal; since $x_iy_i \in O(\cM)^\star$ for all $i \in [n]$, we have $x_i(1-x_i) \in S'/D$ for all $i$ and therefore $S'/D$ is the vanishing ideal of a combinatorial code.
	
	Next we check that if $\phi: S_{\cM_1} \to S_{\cM_2}$ is a strong monomial map, then $\sfD\phi$ is a monomial map of neural rings.
	By definition, $\phi$ induces a monomial map $\F_2[x_1,\dots,x_{n_1}] \to \F_2[x_1,\dots,x_{n_2}]$, sending each $x_i$ to some $x_j$ or 0 as appropriate.
	So, we only need to check this is a well-defined ring homomorphism.
	This follows from properties of polarization: $x^\s(1-x)^\t \in (O(\cM)^\star + D)/D$ if and only if $x^\s y^\t \in O(\cM)^\star$ \cite{gunturkun2017polarization}. Therefore,
		\begin{align*}
			x^\s(1-x)^\t \in (O(\cM)^\star + D)/D & \implies x^\s y^\t \in O(\cM_1)^\star \implies \phi(x^\s y^\t) \in O(\cM_2)^\star \\
			&\implies \sfD\phi(x^\s(1-x)^\t) \in (O(\cM_2)^\star + D)/D.
		\end{align*}
	Thus, $\sfD\phi$ is a well-defined monomial map.
	
	To complete the proof $\sfD$ is a functor, we need to show $\sfD$ respects the identity and composition of morphisms.
	These are immediate from the definitions: $\sfD\operatorname{id}x_i = x_{\operatorname{id}(i)}$ and $\sfD(\phi_f \circ \phi_g)(x_i) = x_{f\circ g(i)} = \sfD\phi_f \circ \sfD\phi_g x_i$.
\end{proof}

\begin{prop}\label{prop:D_commutes}
	The diagram \eqref{eq:functordiagram} commutes.
\end{prop}

\begin{proof}
	We will show that: \begin{enumerate}
	\item for any acyclic oriented matroid $\cM$,
		\[(\sfD \circ \sfS) (\cM) = (\sfR \circ \sfW^+) (\cM), \text{ and}\]
	\item for a strong map of acyclic oriented matroids
		$f : \cM_1 \to \cM_2$, \[(\sfD \circ \sfS)( f) = (\sfR \circ \sfW^+)( f).\]
	\end{enumerate}

\noindent (1)
	We prove the first part by showing that the ring of functions on $\sfW^+\cM$
	is precisely the ring $(\sfD \circ \sfS)(\cM)$. We do this by showing that they are both
	quotients of $\F_2[\bx]$ by the same ideal.
	For a tope $W \subseteq \pm[n]$, denote $\bar{\mathfrak{p}}(W) = \langle x_i \mid W_i = +\rangle + \langle 1-x_i \mid W_i = -\rangle$, i.e.\ the image of $\mathfrak{p}(W)$ under the map $y_i \mapsto 1- x_i$ (recall \cref{eq:OMstar}).
	Then we have
		\[ (\sfD \circ \sfS) (\cM) \cong \F_2[\bx] / \left(\bigcap_{W\in\cW(\cM)} \bar{\mathfrak{p}}(W)\right). \]
	
	Now consider $(\sfR \circ \sfW^+)\cM = \F_2[\bx]/I_{\sfW^+\cM}$. For each tope $W$, let $\mathfrak{m}(W) = \langle x_i \mid W_i = -\rangle + \langle 1-x_i \mid W_i = +\rangle$, the maximal ideal of $\F_2[\bx]$ vanishing at codeword $W^+$.
	As the vanishing ideal of a finite variety, we have
		\[ I_{\sfW^+\cM} = \bigcap_{W \in \cW(\cM)} \mathfrak{m}(W). \]
	By construction, $\mathfrak{m}(W) = \bar{\mathfrak{p}}(-W)$. By symmetry (axiom (V\ref{axiom:symmetryV})), $W$ is a tope if and only if $-W$ is a tope. Therefore, the ideals are defined by the same intersection and therefore the corresponding quotients are identical.

%
%

	\vspace{2mm}
	
	\noindent (2)	Now we prove that strong maps point to the same monomial map via
	$\sfD \circ \sfS$ and $\sfR \circ \sfW^+$. It is sufficient to check the action of
	each monomial map on generators of $\F_2[\bx]$.

	A strong map $\phi_f$ is defined by a set map $f: E_1 \to E_2$ satisfying
	$S \subseteq E_2$ convex
	implies $f^{-1}(S) \subseteq E_1$ convex. The strong monomial map $\sfS \phi_f $
	sends $x_i$ to $0$ if $f(i) = \lop$ and $x_{f(i)}$ otherwise; it acts similarly on $y_i$.
	Applying $\sfD$, the monomial map $(\sfD \circ \sfS)( \phi_f)$ still sends $x_i$ to $0$
	if $f(i) = \lop$ and $x_{f(i)}$ otherwise.

	Going around the diagram the other way, $\sfW^+ \phi_f$ sends a
	codeword $\sigma \in \sfW^+ \cM_2 $ to $f^{-1}(\sigma) \in \sfW^+ \cM_1$. The
	functor $\sfR$ sends a morphism of codes $g:\scC_1 \to \scC_2$ to the ring homomorphism
	given by sending $\nu \in \sfR \scC_2$ to its precomposition with $g$, i.e. $\nu \circ g \in
	\sfR \scC_1$. Starting with a strong map $\phi_f$, let us consider the action of
	$\sfR \sfW^+ \phi_f$ on generators of $\sfR \sfW^+ \cM_1$:
	\[(\sfR \sfW^+\phi_f)(x_i) = x_i \circ [(\sfW^+\phi_f)^{-1}] =
	x_i \circ [\sigma \mapsto f^{-1}(\sigma)]
	\]
	This function takes as input a codeword $\sigma \in \sfW^+ \cM_2$. If $i \in f^{-1}(\sigma)$,
	then it takes the value $1$, and if $i \notin f^{-1}(\sigma)$ then it takes the value $0$.
	If $f(i) = \lop$, then the function is identically zero. If $f(i) \neq \lop$, then
	$x_i \circ [\sigma \mapsto f^{-1}(\sigma)]$ is equal to $x_{f(i)}$, proving that the monomial
	maps $(\sfD \circ \sfS)(\phi_f)$ and $(\sfR \circ \sfW^+)(\phi_f)$ are the same.
\end{proof}

 \Cref{thm:foc} (proven by \cref{P:faithfulnotfull,P:Ifunctor,prop:D_functor,prop:D_commutes}) gives us a new lens to see the neural ideal.
 In essence, neural codes can be seen as a relaxation of oriented matroids.
 The neural ideal is a generalization of the oriented matroid ideal to
 the less constrained category of neural codes.
 Further, \cref{prop:D_functor,prop:D_commutes} demonstrates that the duality between a neural code and its combinatorial relations is analogous to the duality between topes and circuits. In particular, in the special case when a neural code arises from an oriented matroid, the codewords correspond to topes and the elements of the canonical form correspond to circuits.  \cref{L:CFcircuitaxioms}, which states that the elements of the canonical form partially follow the circuit axioms, strengthens this analogy. 

\section{Open questions}\label{S:questions}

The preceding sections have presented our case for employing
oriented matroid theory in the study of neural codes.
However, we stand at the very beginning of exploring this connection. In this
section, we outline some directions for future work. 

\subsection{Is the missing axiom of convex neural codes also lost forever? }

Oriented matroids capture much of combinatorial structure of hyperplane arrangements in a few  axioms. 
Is it possible to give a similar characterization for convex neural codes, or at least for codes which lie below oriented matroids in $\pcode$?
 While general neural codes are not required to satisfy any axioms,
the codes below oriented matroids may be more tractable to combinatorial description.
\begin{question}
  Can the class of neural codes below oriented matroids be characterized
  by a set of combinatorial axioms?
\end{question}

The functorial view we introduced in Section \ref{sec:algebra} may be helpful in finding this characterization. 
If this question is answered in the affirmative, then these codes can be thought of as
``partial oriented matroids.''  
Suppose that $\scC\subseteq 2^{[n]}$ is a code and
$\cM$ is an oriented matroid on ground set $[N]$ such that $\scC = f(\sfL^+\cM)$; then,
we obtain constraints on the set of covectors of $\cM$.  
Each included codeword $\sigma \in\scC$ implies existence of a preimage covector in $\cM$, and
each excluded codeword $\tau\notin\scC$ implies a set of forbidden covectors which may not be in $\cM$.
The oriented matroids satisfying these constraints can then be said to
be ``completions'' of the partial oriented matroid.

Just as we wish to characterize codes lying below oriented matroids with a set of combinatorial axioms, we might also wish to characterize convex codes using a set of combinatorial axioms. However, this is likely not possible. 
In \cite{mayhew2018yes}, Mayhew, Newman, and Whittle show that ``the missing axiom of matroid theory is lost forever." Slightly more formally, they show that there is no sentence characterizing representability in the monadic second order language $MS_0$, which is strong enough to state the standard matroid axioms. Roughly, this means that there is no ``combinatorial" characterization of representability, or no characterization of representability in the language of the other matroid axioms. 

Because we have found strong connections between representability and convexity, it is natural to ask whether a similar statement can be proven for convex codes. 

\begin{question}
Is there a natural language in which we can state ``combinatorial" properties of neural codes, in analogy with the $MS_0$ for matroids? If so, is it possible to characterize convexity in this language?
\end{question}

\subsection{Computational questions}
While we have shown that the convex code decision problem is $\exists \R$-hard,  we have not actually shown that the convex code decision problem lies in $\exists \R$, or is even algorithmically decidable. 
A similar problem, that of determining whether a code has a good cover realization, is undecidable by \cite[Theorem 4.5]{chen2019neural}. 
Here, the distinction between codes with good cover realizations and convex realizations may be significant. 
For instance, while there is an algorithm to decide whether, for any given $d$, a simplicial complex is the nerve of convex open subsets of $\R^d$, for each $d \geq 5$, it is algorithmically undecidable whether a simplicial complex is the nerve of a good cover in $\R^d$ \cite{tancer2013nerves}. 

We outline a possible path towards resolving \cite[Question 4.5]{chen2019neural}, which asks whether there is an algorithm which decides whether a code is convex. 
Our approach hinges on \cref{thm:polytope_matroid}: a code is polytope convex if and only if it lies below a representable oriented matroid.
A first step towards solving the convex code decision problem is answering the following open question:

\begin{question}\label{Q:convexiffpolytope}
Can every convex code be realized with convex polytopes?
\end{question}

In dimension two, this question has an affirmative answer, as a result, the class of planar convex codes is decidable \cite{bukh2022planar}. 
If this holds in all dimensions,
then our Theorem \ref{thm:polytope_matroid} becomes strengthened to the following: 

\begin{conj*}
  A code $\scC$ is convex if and only if $\scC \leq \sf{L}^+\cM$ for $\cM$
  a representable oriented matroid. 
\end{conj*}

If this conjecture holds, then we can replace the problem of
determining whether a code is convex with the problem of determining
whether a code lies below a representable matroid. 
We only need to consider the set of minor-minimal matroids lying above the code, and
check these matroids for representability.

\begin{question}\label{Q:OMsabovecode}
  Given a code $\scC$, is the set of minor-minimal matroids above $\scC$ finite?
  If so, is there an efficient algorithm to enumerate them?
\end{question}

One way to find oriented matroids above a code $\scC$ is to travel step-by-step up the poset $\pcode$. 
While there is a straightforward algorithm to enumerate the $O(n)$
codes which are covered by a code $\scC\subseteq 2^{[n]}$ in $\mathbf{P}_{\mathbf{Code}}$
\cite{jeffs2019sunflowers}, we do not know of a straightforward way to characterize the codes which cover $\scC$. 
If we can characterize these codes as well, we may be able to find a way to ``climb up''
towards an oriented matroid. Alternatively, we can use the ``partial oriented matroid''
perspective described above to obtain a set of constraints that must be obeyed by any oriented
matroid above this code. Then we can look for a matroid satisfying these constraints.

Both of these approaches depend on the minimal size of the
ground set of oriented matroids that lie above $\scC$ in $\pcode$. 
Let 
\[ M(n) = \max_{\substack{\scC \subseteq 2^{[n]}
 \\ \scC \mbox{   \scriptsize  below an}\\\mbox{  \scriptsize oriented matroid}}} \left[ \min_{\scC\leq \sf{L}^+(\cM)} |E(\cM)|\right] \]
 be the smallest $N$ such that any code $\cC$ on $n$ neurons which lies below an oriented matroid lies below an oriented matroid with ground set of size at most $N$. 
Similarly, let
\[ H(n) = \max_{\substack{\scC \subseteq 2^{[n]}
 \\ \scC \mbox{   \scriptsize  below a representable}\\\mbox{  \scriptsize oriented matroid}}} \left[ \min_{\substack{\scC\leq \sf{L}^+\cM \\ \cM \mbox{   \scriptsize  representable}}} |E(\cM)| \right] \] be the smallest $N$ such that any code $\scC$ on $n$ neurons below a representable oriented matroid lies below a representable oriented matroid with ground set of size at most $N$. 
Clearly, $M(n) \leq H(n)$, since any representable matroid is a matroid. 

\begin{question}
Describe the growth of $M(n)$ and $H(n)$ as functions of $n$. Are they equal?
\end{question}

Note that if $H(n)$ is a computable function of $n$, and Question ~\ref{Q:convexiffpolytope}
is answered in the affirmative, then the convex code decision problem is decidable. 

\subsection{Other questions in geometric combinatorics}

Many classic theorems about convex sets, such as Helly's theorem, Radon's theorem, and Caratheodory's theorem, have oriented matroid analogues. In some way, we can view our Theorem \ref{thm:bad_sunflower} as an oriented matroid version of Jeffs' sunflower theorem \cite[Theorem 1.1]{jeffs2019sunflowers}. The fact that the non-convex codes constructed from the sunflower theorem do not lie below oriented matroids shows us that there is some fact about oriented matroids underlying the sunflower theorem.

\begin{question}
Is there a natural oriented matroid version of Jeffs' sunflower theorem? 
\end{question}

Proposition~\ref{prop:no_local} stated that if $\cM$ is an oriented matroid,
the code $\sf{L}^+\cM$ has no local obstructions. That is, for any
$\sigma \in \Delta(\sf{L}^+\cM)\setminus \sf{L}^+\cM$,
$\link_\sigma( \Delta(\sf{L}^+\cM))$ is contractible.
This result can also be found in \cite{edelman2002convex},
where is is phrased as a result about the simplicial complex $\Delta_{\mathrm{acyclic}}(\cM)$.
Something stronger holds for representable oriented matroids: by \cite[Theorem 5.10]{chen2019neural},
if $\cM$ is a {\bf representable} oriented matroid, and
$\sigma \in \Delta(\sf{L}^+\cM)\setminus \sf{L}^+\cM$, then
$\link_\sigma(\sf{L}^+\cM)$ must be collapsible.
Expanding upon this work, \cite{jeffs2019convex} gives stronger conditions
that the link of a missing codeword in a convex code must satisfy. 

We ask whether this holds for all oriented matroids:

\begin{question}
  If $\cM$ is an oriented matroid, and $\sigma \in \Delta(\sf{L}^+\cM)\setminus \sf{L}^+\cM$,
  is $\link_\sigma( \Delta(\sf{L}^+\cM))$ collapsible? More generally, which simplicial complexes can arise as $\link_\sigma( \Delta(\sf{L}^+\cM))$ for $\sigma \in \Delta(\sf{L}^+\cM)\setminus \sf{L}^+\cM$? 
\end{question} 

If not, then the non-collapsibility of $\link_\sigma( \Delta(\sfL^+\cM))$ gives a new ``signature" of non-representability. 

\subsection{Functorial questions}

The maps $\sfW^+$ and $\sfL^+$ established analogies
between structures of oriented matroids and neural codes.
Topes and covectors are translated into the codewords,
and signed circuits are mapped to the combinatorial relations. This leads us to the following
natural question:

\begin{question}
  Do $\sfW^+$ and $\sfL^+$ map other matroid features to meaningful structures
  associated to neural codes? In particular, do the chirotope, rank function, and
  convex closure function have a natural interpretation when mapped to general neural codes?
  
\end{question}

This paper focused on the category of oriented matroids, since
they have a well-established notion of morphisms (strong maps) and since they
are extensively studied. However, there is also a notion of ``affine strong maps''
defined in \cite{hochstattler1999linear} that may serve to turn affine oriented matroids
into a category. This might also admit a natural functor to neural codes. Additionally,
the recently defined objects COM's (which stands for both {\em conditional oriented matroids}
and {\em complexes of oriented matroids}) \cite{bandelt2018coms} are a natural place to try to
extend strong maps next.

\begin{question}
  Can affine oriented matroids with affine strong maps be embedded in $\Code$?
  Can strong maps be defined for COM's
  in such a way that the resulting category can be embedded in $\Code$?
  \end{question}

While strong maps are more frequently used as morphisms of oriented matroids,
weak maps are the next best option.

\begin{question}
  
  Can the category of oriented matroids with morphisms given by weak maps be embedded in $\Code$?
  
  \end{question}


\section*{Acknowledgements}

The authors thank Carina Curto, Emanuele Delucchi, Boaz Haberman,
Gregory Henselman-Petrusek,
Vladimir Itskov, R. Amzi Jeffs, Anne Shiu, Bernd Sturmfels,
and Nora Youngs for helpful conversations. We remark that Henselman-Petrusek
developed connections between neural codes and oriented matroids independently, and we look forward to his upcoming article on the topic.
They also extend deep gratitude to Anne Shiu, R. Amzi Jeffs, and Carina Curto for their comments on an early draft.
Further, they also thank Laura Anderson for heroic efforts to fix the proof of an earlier version of Proposition 4.
They also appreciate the thorough reviews by anonymous referees which contributed
significantly to readability and accuracy of the results.
CL was supported by NSF fellowship grant DGE1255832.
ABK was supported by an NLM Training Program fellowship T15LM007093

	\bibliographystyle{GLG-bibstyle}
	\bibliography{CT-submission}
\end{document}